\documentclass[a4paper]{amsart}
\usepackage{amsmath}
\usepackage{amssymb}
\usepackage{amsfonts}
\usepackage{pxfonts}
\usepackage[OT2,T1]{fontenc}

\usepackage{pict2e}

\usepackage[
textwidth=14.5cm, 
textheight=21cm,
hmarginratio=1:1,
vmarginratio=1:1]{geometry}

\newtheorem{thm}{Theorem}[section]
\newtheorem{cor}[thm]{Corollary}
\newtheorem{lem}[thm]{Lemma}
\newtheorem{prop}[thm]{Proposition}
\theoremstyle{definition}
\newtheorem{defn}[thm]{Definition}

\newtheorem{prob}[thm]{Problem}
\theoremstyle{remark}
\newtheorem{rem}[thm]{Remark}
\newtheorem{opprob}[thm]{Open problem}
\numberwithin{equation}{section}
\numberwithin{figure}{section}

\newcommand{\diff}{\mathrm{d}}
\newcommand{\C}{{\mathbb C}}
\newcommand{\R}{{\mathbb R}}
\newcommand{\D}{{\mathbb D}}
\newcommand{\Te}{{\mathbb T}}
\newcommand{\Eop}{{\mathbf{E}}}

\newcommand{\imag}{\mathrm{i}}

\newcommand{\Mop}{{\mathbf M}}
\newcommand{\Lop}{{\mathbf L}}

\newcommand{\Iop}{\mathbf{I}}

\begin{document}

\title{Weighted integrability of polyharmonic functions}

\author{Alexander Borichev}

\address{Borichev: Laboratoire d'analyse, topologie, probabilit\'es\\
CMI, Aix-Marseille Universit\'e\\
39, rue Fr\'ed\'eric Joliot Curie\\
F--13453 Marseille CEDEX 13\\
FRANCE}

\email{borichev@cmi.univ-mrs.fr}

\author{Haakan Hedenmalm}
\address{Hedenmalm: Department of Mathematics\\
KTH Royal Institute of Technology\\
S--10044 Stockholm\\
Sweden}

\email{haakanh@math.kth.se}

 
\thanks{The research of the first author was
partially supported by the ANR grant FRAB. The second author was 
supported by the G\"oran Gustafsson Foundation (KVA) and Vetenskapsr\aa{}det 
(VR)}

\begin{abstract} 
To address the uniqueness issues associated with the Dirichlet problem for 
the $N$-harmonic equation on the unit disk $\D$ in the plane, we investigate 
the $L^p$ integrability of $N$-harmonic functions with respect to the 
standard weights $(1-|z|^2)^{\alpha}$. The question
at hand is the following. If $u$ solves $\Delta^N u=0$ in $\D$, where 
$\Delta$ stands for the Laplacian, and 
\[
\int_\D|u(z)|^p (1-|z|^2)^{\alpha}\diff A(z)<+\infty,
\]
must then $u(z)\equiv0$? Here, $N$ is a positive integer, $\alpha$ is real, 
and $0<p<+\infty$; $\diff A$ is the usual area element. The answer will, 
generally speaking, depend on the triple $(N,p,\alpha)$. The most 
interesting case is $0<p<1$. For a given $N$, we find an explicit critical 
curve $p\mapsto\beta(N,p)$ -- a piecewise affine function -- 
such that for $\alpha>\beta(N,p)$ there exist non-trivial functions $u$ with 
$\Delta^N u=0$ of the given integrability, while for $\alpha\le\beta(N,p)$, 
only $u(z)\equiv0$ is possible. 
We also investigate the obstruction to uniqueness for the Dirichlet 
problem, that is, we study the structure of the functions in 
$\mathrm{PH}^p_{N,\alpha}(\D)$ when this space is nontrivial. We find a 
new structural decomposition of the polyharmonic functions -- the
cellular decomposition -- which decomposes the polyharmonic weighted $L^p$
space in a canonical fashion. Corresponding to the cellular expansion
is a tiling of part of the $(p,\alpha)$ plane into cells. 
%A particularly 
%interesting collection of cells form the entangled region.

The above uniqueness for the Dirichlet problem may be considered for any 
elliptic operator of order $2N$. However, the above-mentioned 
critical integrability curve will depend rather strongly on the given elliptic 
operator, even in the constant coefficient case, for $N>1$. 
%The case of powers of the Laplacian
%in the context of the disk is special as it allows us to find -- based on the
%Green function -- certain kernels of Poisson type which decay (at all but one 
%point) faster than generally expected near the boundary. 
\end{abstract}

\maketitle

\centerline{\em In memory of Boris Korenblum}

\section{Introduction} 

\subsection{Basic notation} Let 
\[
\Delta:=\frac{\partial^2}{\partial x^2}+\frac{\partial^2}{\partial y^2},
\qquad\diff A(z):=\diff x\diff y,
\]
denote the Laplacian and the area element, respectively. Here, $z=x+\imag y$
is the standard decomposition into real and imaginary parts. We let $\C$ denote
the complex plane, while $\D:=\{z\in\C:|z|<1\}$ and $\Te:=
\{z\in\C:|z|=1\}$ denote the unit disk and the unit circle, respectively. 

\subsection{Polyharmonic functions} 
Given a positive integer $N$, a function 
$u:\D\to\C$ is said to be {\em $N$-harmonic} (alternative terminology: 
{\em polyharmonic of degree $N-1$}) if 
\[
\Delta^{N}u=0
\]
holds on $\D$ in the sense of distribution theory. If $u$ is $N$-harmonic for
some $N$, we say that it is {\em polyharmonic}. Clearly, $1$-harmonic 
functions are just ordinary harmonic functions. The $2$-harmonic functions
are of particular interest, and they are usually said to be {\em biharmonic}. 
While the Laplacian $\Delta$ is associated with a {\em membrane}, 
the bilaplacian $\Delta^2$ is associated with a {\em plate}. There is a 
sizeable literature related to the bilaplacian, and more generally the
$N$-laplacian $\Delta^N$; see, e.g., the books \cite{ACL}, 
\cite{Gak}, \cite{Gara}, \cite{GGS}, and the papers \cite{Gara1}, 
\cite{Hed1}, \cite{Hed2}, \cite{HJS}, \cite{AH}, \cite{GR}, \cite{MM}, 
\cite{Olof1}. 
We could also mention that the biharmonic Green function was used in 
\cite{DKSS} to generalize the factorization of $L^2$ Bergman space functions 
found in \cite{Hed0} to the $L^p$ setting (see also \cite{Hed0.5}, \cite{ARS}).
Later, it was applied to Hele-Shaw flow on surfaces \cite{HedShi02}, 
\cite{HedPer}, \cite{HedOlo}.

\subsection{Uniqueness for the Dirichlet problem}
The Dirichlet problem associated with the $N$-Laplacian is
\[
\begin{cases}
\Delta^{N}u=0\quad \text{on}\,\,\,\D,
\\
\partial_{\mathrm{n}}^j u=f_j\quad\text{on}\,\,\,\Te\, \text{ for}\,\,\,
j=0,1,\ldots,N-1,
\end{cases}
\label{eq-Dirprob}
\]
where $\partial_{\mathrm{n}}$ stands for the (interior) normal derivative.
A natural way to interpret this problem is to first construct a function $F$
on $\D$ which encodes the boundary information from the data $f_0,\ldots,
f_{N-1}$, and to say that \eqref{eq-Dirprob} asks of $u$ that $\Delta^Nu=0$
and that $u-F$ belongs to a class of functions that decay rapidly to $0$ 
near the boundary $\Te$. Often, this decay is understood in terms of Sobolev
spaces (see any book on partial differential equations; for a slightly
different approach, see, e.g., \cite{Dahlb}). A perhaps simpler requirement 
is to ask that  
\begin{equation}
|u(z)-F(z)|=\mathrm{o}((1-|z|)^{N-1})\quad \text{as}\,\,\,|z|\to1^-,
\label{eq-brydecay1}
\end{equation}
and it is easy to see that $u$ is uniquely determined by the differential
equation $\Delta^N u=0$ combined with \eqref{eq-brydecay1}. If we focus on
uniqueness, then by by forming differences we may as well assume $F(z)\equiv0$.
We may then think of \eqref{eq-brydecay1} as
\begin{equation}
(1-|z|)^{-N+1}|u(z)|\in L^\infty_0(\D),
\label{eq-brydecay2}
\end{equation} 
where $L^\infty_0(\D)$ stands for the closed subspace of $L^\infty(\D)$ consisting
of functions with limit $0$ along $\Te$. So \eqref{eq-brydecay2} forces an
$N$-harmonic function to vanish identically. What would happen if we
replace $L^\infty_0(\D)$ by, for instance $L^p(\D)$, for some $p$, $0<p<+\infty$?
We would expect that we ought to change the exponent in the distance to 
the boundary, but by how much? More precisely, we would like to know for which
real (negative) $\alpha$ the implication
\begin{equation}
(1-|z|)^{\alpha}|u(z)|\in L^p(\D)\,\, \implies\,\, u\equiv 0
\label{eq-brydecay3}
\end{equation} 
holds for all $N$-harmonic functions $u$. Here, we shall obtain the 
complete answer to this question.  We obtain an explicit expression 
$\beta(N,p)$ such that the implication \eqref{eq-brydecay3} holds if and only
if $\alpha\le\beta(N,p)/p$. When $\alpha>\beta(N,p)/p$, when we have 
non-uniqueness, we study the source of obstruction to uniqueness (see Theorems
\ref{thm-main2} and \ref{thm-main1}). We obtain an
understanding of those obstructions in terms of a decomposition of $N$-harmonic
functions -- which we call the {\em cellular decomposition} --
associated with a nontrivial factorization of the $N$-Laplacian. This cellular
decomposition is of course related to the classical Almansi representation, 
but the terms are mixed up in a complicated way, optimal for the 
analysis of the boundary behavior.    

\subsection{The local uniqueness problem}
\label{subsec-local1}
Let $J\subset\Te$ be a proper closed arc with positive length, and let 
\begin{equation}
Q(J):=\{z\in\D\setminus\{0\}:\,z/|z|\in J\}
\label{eq-QJ}
\end{equation}
be the corresponding angular triangle. For a subset $E$ of $\C$, we write 
$1_E$ for the characteristic function of $E$, which equals $1$ on $E$ and 
vanishes elsewhere.
The local version of the problem \eqref{eq-brydecay3} reads as follows:
for which $\alpha$ does the implication 
\begin{equation}
1_{Q(J)}(z)(1-|z|)^{\alpha}|u(z)|
\in L^p(Q(J))\,\, \implies\,\, u\equiv 0  
\label{eq-brydecay4.1}
\end{equation}
hold for all $N$-harmonic functions on $\D$? The study of 
\eqref{eq-brydecay4.1} can be thought of as a weighted integrability version 
of Holmgren's uniqueness problem (see, e.g., John's book \cite{John}).
We completely resolve this question as well: the implication in 
\eqref{eq-brydecay4.1} holds if and only if $\alpha\le -2N+1
-\frac{1}{p}$ (see Theorem \ref{thm-8.1.1}). 
Note that the answer does not depend on the length of the
arc $J$. In addition it is worth observing that $\beta(N,p)/p=-2N+1
-\frac{1}{p}$ for $p$ in the interval $0<p\le 1/(2N)$, so for such small $p$
the global (Dirichlet-type) and local (Holmgren-type) problems have the same
uniqueness criteria. 
%A reasonable interpretation is that for small $p$ the singularities
%are less important 
For other domains such as the interior of an ellipse the critical range of 
$\alpha$ will be different. 
%What happens then (at least for $N=2$) is that 
%the examples which give criticality cease to exist; 
See Subsections \ref{subsec-otherdom} and \ref{subsec-8.extra} for details.  

\subsection{Remarks on the setting}
It is natural to ask whether our results generalize to other domains and to
more general elliptic partial differential operators. After all, our methods
are quite specific to the $N$-Laplacian on circular disks 
(half-planes would work as well). It is possible to show that already 
replacing the circle by an ellipse while keeping the bilaplacian $\Delta^2$ 
(for $N=2$) changes the problem considered here, so that it obtains a 
different answer; see Subsections \ref{subsec-otherdom} and 
\ref{subsec-8.extra} for details (cf. \cite{Hed3}). We remark here the 
appearance of a connection with the theory of quadrature domains.  In 
conclusion, we use specific methods because the problem needs them. 

\subsection{Acknowledgements}
We are grateful to Miroslav Pavlovi\'c for pointing out to us the 
results contained in his papers \cite{Pa2}, \cite{Pa1}.
We also thank Elias Stein for making us aware of the theory of axially
symmetric potentials, and Alexandru Ionescu for suggesting the local 
uniqueness problem. Finally, we thank Dmitry Khavinson for valuable comments.

\section{The Almansi expansion and weighted Lebesgue spaces 
in the disk} 

\subsection{Some additional notation}
\label{subsec-addlnot}

For $z_0\in\C$ and positive real $r$, let $\D(z_0,r)$ denote the open 
disk centered at $z_0$ with radius $r$; moreover, we let $\Te(z_0,r)$ denote
the boundary of $\D(z_0,r)$ (which is a circle). We let $\diff s(z)=|\diff z|$
denote arc length measure on curves (usually on circles such as 
$\Te=\Te(0,1)$).

The complex differentiation operators 
\[
\partial_z:=\frac{1}{2}\bigg(\frac{\partial}{\partial x}-\imag 
\frac{\partial}{\partial y}\bigg)
,\qquad \bar\partial_z:=\frac{1}{2}\bigg(\frac{\partial}{\partial x}+\imag 
\frac{\partial}{\partial y}\bigg),
\]
will be useful. Here, $z=x+\imag y$ is the standard decomposition of a complex
number into real and imaginary parts. It is easy to check that 
$\Delta=4\partial_z\bar\partial_z$. If we let 
\[
\nabla=\nabla_z:=\bigg(\frac{\partial}{\partial x},
\frac{\partial}{\partial y}\bigg)
\]
denote the gradient, then we find that
\begin{equation}
|\nabla_z u|^2=2\big(|\partial_z u|^2+|\bar\partial_z u|^2\big).
\label{eq-nabla2}
\end{equation}

\subsection{The Almansi expansion and the extension of a 
polyharmonic function}
\label{subsec-Almansi1}
The classical {\em Almansi expansion} (or {\em Almansi representation}) 
asserts that $u$ is $N$-harmonic if and only if it is of the form
\begin{equation}
u(z)=u_0(z)+|z|^2u_1(z)+\cdots+|z|^{2N-2}u_{N-1}(z),
\label{eq-Almansi1}
\end{equation}
where all the functions $u_j$ are harmonic in $\D$ (see, e.g., Section 32 of 
\cite{Gak}). The harmonic functions 
$u_j$ which appear in the Almansi expansion \eqref{eq-Almansi1} are uniquely
determined by the given $N$-harmonic function $u$, as is easy to see from
the Taylor expansion of $u$ at the origin, by appropriate grouping of the 
terms. This allows us to define uniquely the {\em extension operator} 
$\mathbf{E}$:
\begin{equation}
\mathbf{E}[u](z,\varrho)=u_0(z)+\varrho^2u_1(z)+\cdots+
\varrho^{2N-2}u_{N-1}(z);
\label{eq-Almansiext}
\end{equation}
the function $\mathbb{E}[u]$ will be referred to as {\em the extension} of $u$.
The extension $\mathbb{E}[u](z,\varrho)$ has the following properties: 

(a) it is an even polynomial of degree $2N-2$ in the variable $\varrho$, 

(b) it is harmonic in the variable $z$, and
\begin{equation}
\mathbf{E}[u](z,|z|)\equiv u(z).
\label{eq-restr1}
\end{equation} 
In fact, the above properties (a)--(b) and \eqref{eq-restr1} 
characterize the extension operator $\mathbf{E}$. 

As for notation, we write 
$\mathrm{PH}_N(\D)$ for the linear space of all $N$-harmonic function in the 
unit disk $\D$.

\subsection{The standard weighted Lebesgue spaces}
Let $u:\D\to\C$ be a (Borel) measurable function. Given reals $p,\alpha$ with
$0<p<+\infty$, we consider the Lebesgue space $L^p_\alpha(\D)$ of
(equivalence classes of) functions $u$ with
\begin{equation}
\|u\|^p_{p,\alpha}:=\int_\D |u(z)|^p(1-|z|^2)^{\alpha}
\,\diff A(z)<+\infty.
\label{eq-norm1}
\end{equation}
These spaces are standard in the Bergman space context \cite{HKZ}. For
$1\le p<+\infty$, they are Banach spaces, and for $0<p<1$, they are
quasi-Banach spaces. 
Clearly, we have the 
%{\ {red} 
inclusion
%}%containment 
\begin{equation}
L^p_{\alpha}(\D)\subset L^p_{\alpha'}(\D)\quad\text{for}\,\,\,\alpha<\alpha'.
\label{eq-contain1}
\end{equation}

\subsection{The $L^p$-type of a measurable function}
We use the standard weighted Lebesgue spaces to define the concept of
the $L^p$-type of a function.  

\begin{defn} $(0<p<+\infty)$
For a Borel measurable function $u:\D\to\C$, let its $L^p$-{\em type} be 
the number
\[
\beta_p(u):=\inf\{\alpha\in\R:\,u\in L^p_\alpha(\D)\},
\]
if the infimum is taken over a non-empty collection. If instead $u\notin
L^p_\alpha(\D)$ for every $\alpha\in\R$, we write 
$\beta_p(u):=+\infty$.
\end{defn}

The $L^p$-type measures the boundary growth or decay of the given function 
$u$. In particular, it is rather immediate that if $u$ has compact support 
in $\D$, then its $L^p$-type equals $\beta_p(u)=-\infty$ for all $p$,  
$0<p<+\infty$. It is a consequence of H\"older's inequality that for a fixed 
$u$, the function $p\mapsto\beta_p(u)$ is convex (interpreted liberally).

Here, we want to study the $p$-type in the context of the spaces of
$N$-harmonic functions $\mathrm{PH}_N(\D)$. If we think of the elements of the
space $\mathrm{PH}_N(\D)$ as physical states, we may interpret the $L^p$-type
$\beta_p(f)$ as the ``$p$-temperature'' of the state $f\in\mathrm{PH}_N(\D)$. 
If we fix $N$ and freeze the system -- i.e., we  consider only states of low 
``$p$-temperature'' -- we should expect that the degrees of freedom are 
reduced, and eventually, only the trivial state $0$ would remain. 
The ``$p$-temperature'' at which this transition occurs is the critical 
``$p$-temperature'' for the given $N$.    

In mathematical terms, we shall be concerned with the following problem. 
 
\begin{prob}
Given a positive integer $n$ and a $p$ with $0<p<+\infty$, what is the value
of
\begin{equation}
\beta(N,p):=\inf
\big\{\beta_p(f):\,f\in\mathrm{PH}_N(\D)\setminus\{0\}\big\}\,\,?
\label{eq-betapn}
\end{equation}
In other words, what is the smallest possible $L^p$-type of a non-trivial
function $f\in\mathrm{PH}_N(\D)$?
Moreover, is the above infimum attained (i.e., is it a minimum)?
%For which values of $(n,p,\alpha)$ is $\mathrm{PH}_{n,\alpha}^p(\D)\ne\{0\}$?
%In other words, when does there exist a non-trivial $n$-harmonic function in
%$L^p_\alpha(\D)$?
\label{prob-1}
\end{prob} 

We call the function $p\mapsto\beta(N,p)$ the {\em critical integrability type 
curve for the $N$-harmonic functions}, and the function 
$(N,p)\mapsto\beta(N,p)$ 
the {\em critical integrability type curves for the polyharmonic functions}.

The notion of the critical integrability type is rather parallel to 
Makarov's integral means spectrum in the context of bounded univalent 
functions \cite{Mak} (see also \cite{HedShi05}); there, however, a ``sup'' is
used instead of an ``inf'' in the formula analogous to \eqref{eq-betapn}, so 
Makarov's integral means spectrum is automatically convex, which is not true
about the critical integrability type curve (see remark below). 

\begin{rem}
(a) Since we take an infimum over a collection of $f$, the property that
$p\mapsto \beta_p(f)$ is convex does not carry over to $p\mapsto \beta(N,p)$;
indeed, examples will show that $p\mapsto \beta(N,p)$ fails to be convex.

(b) If we were to replace the ``inf'' with a ``sup'' in the above 
definition \eqref{eq-betapn}, we would not obtain an interesting concept, 
as it is easy to construct a harmonic function $f$ on $\D$ with 
$\beta_p(f)=+\infty$ for all $p$, $0<p<+\infty$.
%Let us fix the positive integer $n$. In view of \eqref{eq-contain1}, the
%set
%\[
%\mathcal{A}_n(p):=
%\big\{\alpha\in\R:\,\mathrm{PH}_{n,\alpha}^p(\D)\ne\{0\}\big\}
%\]
%has the property that 
%\[
%\forall\alpha,\alpha'\in\R:\,\,\alpha<\alpha'\,\,\,\text{and}\,\,\,
%\alpha\in\mathcal{A}_n(p)\,\,\,\,
%\Longrightarrow\,\,\,\,\alpha'\in\mathcal{A}_n(p).
%\]
%It is easy to see that $\mathcal{A}_n(p)\ne\emptyset$ for each positive $p$,
%and a little bit more demanding to show that $\mathcal{A}_n(p)\ne\R$. It now 
%follows that $\mathcal{A}_n(p)\ne\R$ has a very simple form; there exists a 
%real number $\alpha_*(n,p)$ such that   
%\[
%]\alpha_*(n,p),+\infty[\,\,\,\subset\,\mathcal{A}_n(p)\,\subset\,\,
%[\alpha_*(n,p),+\infty[.
%\]
\end{rem}

\subsection{The Almansi expansion and the boundary decay of 
polyharmonic functions}
\label{subsec-Almansi}

The Almansi expansion \eqref{eq-Almansi1} can be expressed in the following 
form:
\begin{equation}
u(z)=v_0(z)+(1-|z|^2)v_1(z)+\cdots+(1-|z|^2)^{N-1}v_{N-1}(z),
\label{eq-Almansi2}
\end{equation}
where all the functions $v_j$ are harmonic in $\D$. In terms of the functions 
$u_j$ of \eqref{eq-Almansi1}, the functions $v_j$ are given as
\[
v_j=(-1)^j\sum_{k=j}^{N-1}\binom{k}{j}u_k.
\]
In view of \eqref{eq-Almansi2}, we might be inclined to believe that the
functions $u$ with 
\[
v_0=v_1=\cdots=v_{N-2}=0,
\] 
that is,
\begin{equation}
u=(1-|z|^2)^{N-1}v_{N-1}(z),
\label{eq-smallest?}
\end{equation}
should be the smallest near the boundary. To our surprise, we find
that this is not true, at least if we understand the question in terms of
Problem \ref{prob-1} with $0<p<\frac13$. Somehow the functions 
$v_0,\ldots,v_{N-1}$ can cooperate to produce non-trivial functions which 
decay faster than functions of the type \eqref{eq-smallest?}. We think of 
this phenomenon as an {\em entanglement} (see Section \ref{sec-mainres} 
for more details).

\subsection{The standard weighted Lebesgue spaces of 
polyharmonic functions} 
We put
\[
\mathrm{PH}_{N,\alpha}^p(\D):=\mathrm{PH}_N(\D)\cap L^p_\alpha(\D),
\]
and endow it with the norm or quasi-norm structure of $L^p_{\alpha}(\D)$.
This is the subspace of $L^p_\alpha(\D)$ consisting of $N$-harmonic functions. 
It turns out that it is a closed subspace; this is rather non-trivial for 
$0<p<1$, even for $N=1$. Actually, the proof is based on a property which we 
will refer to as {\em Hardy-Littlewood ellipticity of the Laplacian} 
(cf. \cite{HL}; see also \cite{Gar}, pp. 121--123).

\subsection{The harmonic case ($N=1$)}
In \cite{Ale}, Aleksandrov studied essentially our Problem \ref{prob-1}
in the case of $N=1$ (harmonic functions). 
To explain the result, we make some elementary calculations.
The constant function $U_0(z)\equiv1$ is harmonic, and 
\begin{equation}
U_0=1\in\mathrm{PH}^{p}_{1,\alpha}(\D)\,\,\,\Longleftrightarrow \,\,\,
\alpha>-1.
\label{eq-eq1}
\end{equation}
This shows that 
\begin{equation}
\beta_p(U_0)=-1.
\label{eq-betau0}
\end{equation}
Next, we turn to the Poisson kernel at $1$,
\[
U_1(z)=P(z,1)=\frac{1-|z|^2}{|1-z|^2}.
\]
We shall need the following lemma. The formulation involves the standard 
Pochhammer symbol notation $(x)_j:=x(x+1)\cdots(x+j-1)$.

\begin{lem}
Let $a,b$ be two real parameters with $b\ge0$. If we put
\[
I(a,b):=\int_\D\frac{(1-|z|^2)^a}{|1-z|^{2b}}\diff A(z),
\]
then $I(a,b)=+\infty$ if $a\le -1$ or if both $b>0$ and $a\le 2(b-1)$. If, 
on the other hand, $a>-1$ and $b=0$, then $I(a,b)=\pi/(a+1)$. Moreover, if
$a>-1$, $b>0$, and $a>2(b-1)$, then $I(a,b)<+\infty$, with value
\[
I(a,b)=\pi\sum_{j=0}^{+\infty}\frac{[(b)_j]^2}{j!(a+1)_{j+1}}.
\]
\label{lem-int}
\end{lem}

\begin{proof}
By Taylor expansion and polar coordinates, we find that
\begin{multline*}
I(a,b)=\int_\D|1-z|^{-2b}(1-|z|^2)^{a}\diff A(z)
=\int_\D\bigg|\sum_{j=0}^{+\infty}\frac{(b)_j}{j!}z^j\bigg|^{2}
(1-|z|^2)^{a}\diff A(z)
\\
=\sum_{j=0}^{+\infty}\frac{[(b)_j]^2}{[j!]^2}\int_\D|z|^{2j}
(1-|z|^2)^{a}\diff A(z)
=\pi\sum_{j=0}^{+\infty}\frac{[(b)_j]^2}{[j!]^2}\int_0^1 t^{j}
(1-t)^{a}\diff t.
\end{multline*}
The (Beta) integral on the right-hand side diverges for $a\le-1$, so that
$I(a,b)=+\infty$ then. For $a>-1$, the Beta integral is quickly evaluated, 
and we obtain that
\[
I(a,b)=\pi\sum_{j=0}^{+\infty}\frac{[(b)_j]^2}{j!(a+1)_{j+1}},
\end{equation*}
and for $b>0$ the sum on the right-hand side converges if and only if 
$a>2(b-1)$, 
by the standard approximate formulae for the Gamma function. 
\end{proof}

We see from Lemma \ref{lem-int} that 
\begin{equation}
U_1=P(\cdot,1)\in
\mathrm{PH}^{p}_{1,\alpha}(\D)\,\,\,\Longleftrightarrow \,\,\,
\alpha>\max\{p-2,-1-p\},
\label{eq-eq2}
\end{equation}
which implies that its $L^p$-type is 
\begin{equation}
\beta_p(U_1)=\max\{p-2,-1-p\}.
\label{eq-betau1}
\end{equation}

Now, in view of \eqref{eq-betau0} and \eqref{eq-betau1}, the critical 
integrability type $\beta(1,p)$ satisfies
\[
\beta(1,p)\le \min\{\beta_p(U_0),\beta_p(U_1)\}=
\min\big\{-1,\max\{p-2,-1-p\}\big\}.
\]

The 
%{\color{red} 
profound
%} 
work of Aleksandrov \cite{Ale} is mainly concerned with harmonic functions 
in the unit ball of $\R^{d}$ and the maximal possible rate of decay of the
$L^p$-integral on concentric spheres $x_1^2+\cdots+x_{d}^2=r^2$. 
In the rather elementary planar case $d=2$, it gives to the 
following result (see also Suzuki \cite{Suz}).

\begin{thm}
We have that $\beta(1,p)=\min\{-1,\max\{p-2,-1-p\}\}$ for all $p$, 
$0<p<+\infty$. 
\end{thm}

This has the interpretation that the constant function $U_0=1$ 
and the Poisson kernel $U_1=P(\cdot,1)$ are jointly extremal for the problem of
determining the critical $L^p$-type for harmonic functions. 
Indeed, for $0<p\le1$, we have $\beta(1,p)=\beta_p(U_1)$, while 
for $1\le p<+\infty$, we have instead $\beta(1,p)=\beta_p(U_0)$. The 
function $\beta(1,p)$ is therefore continuous and piecewise affine: 
$\beta(1,p)=-1-p$ for $0<p\le\frac12$, $\beta(1,p)=p-2$ for 
$\frac12\le p\le1$, and $\beta(1,p)=-1$ for $1\le p<+\infty$.

\section{Main results}
\label{sec-mainres}

\subsection{Characterization of the critical integrability type curve}

We let $b_j(p)$ be the function
\begin{equation}
b_{j,N}(p):=\max\big\{-1-(j+N-1)p,-2+(j-N+1)p\big\},\qquad j=1,\ldots,N,
\label{eq-bj}
\end{equation}
whose graph is piecewise affine, while for $j=0$ we put
\begin{equation}
b_{0,N}(p):=-1-(N-1)p,
\label{eq-b0}
\end{equation}
which is affine. It is easy to check that
\begin{equation}
b_{j,N+1}(p)+p=b_{j,N}(p),\qquad j=0,\ldots,N.
\label{eq-brel}
\end{equation}
We present our first main theorem.

\begin{thm}
$(0<p<+\infty)$ For $N=1,2,3,\ldots$ and for real $\alpha$, we have that 
\[
\mathrm{PH}^p_{N,\alpha}(\D)=\{0\}\,\,\,\iff\,\,\,
\alpha\le\min_{j:0\le j\le N} b_{j,N}(p).
\]
\label{thm-main2}
\end{thm}

As an immediate consequence, we obtain the explicit evaluation of the critical
integrability type curve.

\begin{thm}[The sawtooth theorem]
The critical integrability type for the polyharmonic functions is given by
\[
\beta(N,p)=\min_{j:0\le j\le N} b_{j,N}(p),
\]
for $0<p<+\infty$ and $N=1,2,3,\ldots$.
\label{thm-main1}
\end{thm}

%Given $n\ge 1$, we define a function $\beta_N$ on $(0,+\infty)$
%by the relations
%$$
%\beta_N(t)=\left\{
%\begin{aligned}
%1+(2n-1)t,&\quad t\le \frac{1}{2n},\\
%1+(2n-k-1)t,&\quad \frac{1}{2(n-k)+1}\le t\le \frac{1}{2(n-k)},\\
%&\qquad\qquad\qquad\qquad\qquad 1\le k\le n-1,\\
%1+(n-1)t,&\quad t\ge 1,
%\end{aligned}
%\right.
%$$
%extending it linearly on the intermediate intervals.
%It is easily seen that 
%\begin{equation}
%\left\{\begin{aligned}
%\beta_{n+1}(t)&=\beta_n(t)+t, \qquad t\ge \frac 1{2n+1},\\
%\beta_{n+1}(t)&\ge\beta_n(t)+t, \qquad \frac 1{2n+2}\le t
%\le \frac 1{2n+1}.
%\end{aligned}
%\right.
%\label{7}
%\end{equation}
%
%We set 
%$$
%\gamma_{n,k}(t)=\min\{1+(n+k-1)t,2-(n-k-1)t\}.
%$$
%Then
%$$
%\beta_n=\max_{0\le k\le n}\gamma_{n,k}.
%$$
%Given $p>0$, $n\ge 1$, $\beta>1$, we set 
%$$
%\mathcal C(n,p,\beta)=\text{card}\,\{k:0\le k\le n,\beta<\gamma_{n,k}(t)\}.
%$$
%\medskip
Being the minimum of a finite number of continuous piecewise affine functions,
the function $p\mapsto\beta(N,p)$ is then continuous and piecewise affine.
It is easy to check that
%Since 
%\[
%-2+(j-N+1)p\ge-1-(N-1)p
%\]
%holds for $1\le j\le N$ and $1\le p<+\infty$, we have that 
\begin{equation}
\beta(N,p)=\min_{j:0\le j\le N}b_{j,N}(p)=\min\{b_{0,N}(p),b_{1,N}(p)\}
\quad\text{for}\,\,\,\tfrac13\le p <+\infty,
\label{eq-beta1infty}
\end{equation}
%and consequently
%\begin{equation}
%\beta(N,p)=b_{0,N}(p)=-1-(N-1)p,\qquad 1\le p<+\infty.
%\label{eq-beta1infty}
%\end{equation}
and since $b_{0,N}(p)$ and $b_{1,N}(p)$ equal the $L^p$-type of the 
$N$-harmonic functions $z\mapsto (1-|z|^2)^{N-1}$ and $z\mapsto(1-|z|^2)^N
/|1-z|^{2}$, we may interpret this as concrete support for
the intuition of Subsection \ref{subsec-Almansi} (based on the Almansi
expansion) for $\frac13\le p<+\infty$. 
However, it is again easy to verify that 
\[
\beta(N,p)=\min_{j:0\le j\le N}b_{j,N}(p)<\min\{b_{0,N}(p),b_{1,N}(p)\}
\quad\text{for}\,\,\,0<p<\tfrac13,
\]
so the intuition fails then, and we interpret this as the appearance 
of entanglement.

\begin{rem}
%{\color{red} 
We draw the graphs of $p\mapsto\beta(N,p)$, for $N=2,3$, in 
Figures \ref{fig-1} ($N=2$) and \ref{fig-2} ($N=3$), respectively.
%}
\end{rem}

We prove Theorem \ref{thm-main2} in Section \ref{sec-6}; the work is based on
the property of the $N$-Laplacian which we call Hardy-Littlewood
ellipticity (see Section \ref{sec-4}). 
%Later, in Section \ref{sec-7} we discuss the structure of 
%the spaces $\mathrm{PH}^p_{N,\alpha}(\D)$ for near-critical 
%values of the parameter $\alpha$. Generally speaking, we find nontrivial
%entanglement for $0<p<\frac13$ and lack of entanglement for 
%$\frac13<p<+\infty$. 

\subsection{The structure of polyharmonic functions}
%\label{sec-7:0}
We pass to the study of the structure of the space $\mathrm{PH}^p_{N,\alpha}
(\D)$ when the space contains nontrivial elements.
We fix an integer $N=2,3,4,\ldots$, and let $\mathcal{A}_N\subset\R^2$ 
be the open set
\[
\mathcal{A}_N:=\big\{(p,\alpha):\,\,0<p<+\infty\,\,\text{and}\,\,
\min_{j:0\le j\le N} b_{j,N}(p)<\alpha\big\};
\]
then the assertion of Theorem \ref{thm-main2} is equivalent to the statement
\[
(p,\alpha)\in\mathcal{A}_N \,\, \iff\,\, \mathrm{PH}^p_{N,\alpha}(\D)\ne\{0\}.
\]
We will at times refer to $\mathcal{A}_N$ as the {\em admissible region}.
%The following result is key to our analysis of the structure of the space 
%$\mathrm{PH}^p_{N,\alpha}(\D)$. 
Let us introduce the second order elliptic partial differential 
operator $\Lop_{\theta}$ indexed by a real parameter $\theta$,  
\begin{equation}
\Lop_{\theta}[u](z):=
(1-|z|^2)\Delta u(z)+4\theta[z\partial_z u(z)+\bar z\bar\partial_z u(z)]
-4\theta^2u(z).
\label{eq-LN-0}
\end{equation}
We note that (when desirable) the complex derivatives can be eliminated by 
considering polar coordinates:
\[
r\partial_r=z\partial_z+\bar z\bar\partial_z.
\] 
We remark here that it is possible to represent $\Lop_\theta$ as 
\[
\Lop_\theta[u]=(1-|z|^2)^{2\theta+1}\nabla\cdot\{(1-|z|^2)^{-2\theta}\nabla u\}
-4\theta^2u, 
\]
and that $\Lop_\theta$ is somewhat analogous to the partial differential
operators considered in the theory of generalized axially symmetric potentials
\cite{Wein} in the context of half-planes.  
To simplify the notation, we let $\Mop$ denote the multiplication operator 
given by
\[
\Mop[v](z):=(1-|z|^2)v(z).
\]
We remark further that the operators $\Lop_{\theta}$ are related to
the operators $D_\alpha$ introduced by Olofsson in \cite{Olof2}: 
$4\Mop^{\alpha+1} D_\alpha=\Lop_{\alpha/2}$. % for $\alpha\in\mathbb Z_+$.
Moreover, we observe that for $j=0,\ldots,N-1$,
\begin{equation}
\Mop^j[v]\in L^p_\alpha(\D)
\,\,\iff\,\,v\in  L^p_{\alpha+jp}(\D),
\label{eq-multspaces1}
\end{equation}
so that
\begin{equation}
\Mop^j[v]\in \mathrm{PH}^p_{N,\alpha}(\D) \,\,\text{and}\,\, \Delta^{N-j}v=0
\,\,\,\iff\,\,\,v\in \mathrm{PH}^p_{N-j,\alpha+jp}(\D).
\label{eq-multspaces2}
\end{equation}

The following is our main structure theorem for the $N$-harmonic functions.

\begin{thm} {\rm(The cellular decomposition theorem)}
%{\color{red}
Let $\alpha$ be real, and let $p$ be positive.
%} 
Then, for $N=1,2,3,\ldots$, every $u\in\mathrm{PH}^p_{N,\alpha}(\D)$ has a 
unique decomposition
\[
u=w_0+\Mop[w_1]+\cdots+\Mop^{N-1}[w_{N-1}],
\]
where each term $\Mop^j[w_j]$ is in $\mathrm{PH}^p_{N,\alpha}(\D)$,
while the functions $w_j$ are $(N-j)$-harmonic and solve the partial 
differential equation $\Lop_{N-j-1}[w_j]=0$ on $\D$, for $j=0,\ldots,N-1$. 
%In addition, for a given $j$,
%we have that $w_j=0$ unless $\alpha>b_{N-j,N}(p)$. 
\label{thm-entamain2}
\end{thm}

\begin{rem}
(a) In the context of Theorem \ref{thm-entamain2}, the mapping 
$u\mapsto \Mop^j[w_j]$ defines an idempotent operator $\mathbf{P}_j$ 
on $\mathrm{PH}^p_{N,\alpha}(\D)$, for $j=0,\ldots,N-1$. It is clear 
from the proof of the theorem that each $\mathbf{P}_j$ acts continuously on
$\mathrm{PH}^p_{N,\alpha}(\D)$, and that $\mathbf{P}_j\mathbf{P}_k=0$ for 
$j\ne k$.

\noindent (b) The above cellular decomposition theorem should be compared 
not only with the alternative Almansi respresentation, but also with a 
result of Weinstein (see \cite{Wein2}, p. 251).
\end{rem}

Though reminiscent of the alternative Almansi expansion \eqref{eq-Almansi2} 
%{\color{red} 
(compare with Pavlovi\'c \cite{Pa1}),
%} 
the expansion of Theorem \ref{thm-entamain2} is different, because here, the 
functions $w_j$ are not assumed harmonic, instead they solve the partial 
differential equation $\Lop_{N-j-1}[w_j]=0$. We suggest to call the 
expansion of Theorem \ref{thm-entamain2} the {\em cellular decomposition},
because it relates to the cell structure of the admissible region 
$\mathcal{A}_N$ (see below).
It is a crucial feature of Theorem \ref{thm-entamain2} that each term of the 
decomposition remains in the space $\mathrm{PH}^p_{N,\alpha}(\D)$. This means
that we may analyze each term separately. 

\subsection{The structure of polyharmonic functions: the 
entangled and unentangled regions}

We consider the following relatively closed bounded subset of the admissible
region $\mathcal{A}_N$
\[
\mathcal{E}_N:=\big\{(p,\alpha)\in \mathcal{A}_N:
\,\,0<p<\tfrac13\,\,\text{and}\,\,
\alpha\le -1-Np\big\}.
\]
We will refer to $\mathcal{E}_N$ as the {\em entangled region}. The 
complement $\mathcal{N}_N:=\mathcal{A}_N\setminus\mathcal{E}_N$ is then open, 
and we call it the {\em unentangled region}. These two regions will be 
further subdivided into smaller units which we refer to as {\em cells}. 
In particular, the unentangled region has a 
{\em principal unentangled cell}, 
\[
\mathcal{N}_N^{(1)}:=\big\{(p,\alpha)\in \mathcal{N}_N:
\,\,\tfrac13<p\,\,\text{and}\,\,
\alpha\le \min\{-2+(3-N)p,-1+(2-N)p\}\big\},
\]
which is relatively closed in $\mathcal{N}_N$. 

\begin{prop}
We have that $(p,\alpha)\in\mathcal{A}_N$ belongs to the entangled region 
$\mathcal{E}_N$ if and only if
\[
u=\Mop^{N-1}[v]\in \mathrm{PH}^p_{N,\alpha}(\D) \,\,\,\text{for harmonic}
\,\,\,v\,\,\,\implies\,\,\,u=0.
\]
\label{prop-ent3.4}
\end{prop}

In other words, the entangled region describes where the space 
$\mathrm{PH}^p_{N,\alpha}(\D)$ contains no non-trivial functions of the form
$(1-|z|^2)^{N-1}v(z)$ with $v$ harmonic. The principal unentangled cell
has a similar-looking characterization.

\begin{prop}
We have that $(p,\alpha)\in\mathcal{N}_N$ belongs to the principal 
unentangled cell $\mathcal{N}_N^{(1)}$ if and only if
\[
u\in\mathrm{PH}^p_{N,\alpha}(\D)\,\,\,\implies\,\,\,
u=\Mop^{N-1}[v]\,\,\,\text{for some harmonic}\,\,\,v.
\]
\label{prop-snent3.5}
\end{prop}

Recent work of Olofsson \cite{Olof2} shows the following. 
%As for notation, let $\mathfrak{H}^p_{\theta,\alpha}(\D)$ denote 
%the space of all functions $v\in L^p_\alpha(\D)$ that solve 
%$\Lop_\theta[v]=0$ 
%on $\D$ in the sense of distribution theory; by ellipticity, we may assume 
%the functions $\mathfrak{H}^p_{\theta,\alpha}(\D)$ are continuous.  
We assume that $\theta$ is a nonnegative integer for technical reasons; 
the result may well be true for general real $\theta>-\frac12$.

\begin{prop}
$(0<p<+\infty)$ Fix a nonnegative integer $\theta$ and a real parameter 
$\alpha$. 
For a function $v\in L^p_\alpha(\D)$, the following are equivalent: 

\noindent{\rm(i)} The function $v$ is continuous in $\D$ and solves 
$\Lop_\theta[v]=0$ in $\D$ in the sense of distributions, and

\noindent{\rm(ii)} The function $v$ is of the form 
\[
v(z)=\frac{1}{2\pi}\int_{\Te}
\frac{(1-|z|^2)^{2\theta+1}}{|1-z\bar\xi|^{2\theta+2}}f(\xi)\diff s(\xi),
\qquad z\in\D,
\]
for some distribution $f$ on $\Te$, where the integral is understood in 
the sense of distribution theory. 
\label{prop-olof1}
\end{prop}

Proposition \ref{prop-olof1} tells us that the functions $w_j$ of Theorem
\ref{thm-entamain2} may expressed as Poisson-type integrals of their 
distributional boundary values (compare with the remark below). 

\begin{rem}
(a)
In the context of Proposition \ref{prop-olof1}, the boundary distribution $f$
equals $c(\theta)$ times the limit of of the dilates $v_r(\zeta)=v(r\zeta)$,
where $\zeta\in\Te$, as $r\to1^-$. Here, $c(\theta)$ is the constant
$$%\begin{equation}
c(\theta):=\frac{\Gamma(1+\theta)^2}{\Gamma(1+2\theta)}.
%\label{eq-ctheta}
$$%\end{equation}

(b) The above theorem of Olofsson should be compared with the work of Huber 
\cite{Hub}. 
\end{rem} 

\subsection{The cells of the admissible region: entangled and
unentangled}
To properly analyze all the cells of the admissible region, we should first 
introduce a modification of the functions $b_{j,N}(p)$ given by \eqref{eq-bj}. 
So, we write, for $j=1,\ldots,N$,  
%{\color{red}
\begin{equation}
a_{j,N}(p):=\min\big\{b_{j,N}(p),-1+(j-N)p\big\}
=
\begin{cases}-1-(j+N-1)p,\qquad 0<p<1/(2j),\\
-2+(j-N+1)p,\qquad 1/(2j)\le p<1,\\
-1+(j-N)p,\qquad p\ge 1,
\end{cases}
\label{eq-aj}
\end{equation}
%}
and observe that $a_{j,N}(p)=b_{j,N}(p)$ for $0<p\le1$, while $a_{j,N}(p)=
-1+(j-N)p$ for $p\ge1$. We check that the graph of $p\mapsto a_{j,N}(p)$ is 
continuous and piecewise affine. 
The analogue of \eqref{eq-brel} reads as follows:
\begin{equation}
a_{j,N+1}(p)+p=a_{j,N}(p),\qquad j=1,\ldots,N.
\label{eq-arel}
\end{equation}
If we draw all the curves $\alpha=a_{j,N}(p)$ 
within the admissible region $\mathcal{A}_N$, they slice up the region into 
pieces we call {\em cells}. 
%These cells are either contained in the 
%entangled or the unentangled regions. Those contained in the 
%entangled region are called {\em entangled cells}, while those contained
%in the unentangled region are called {\em unentangled cells}. 
To make this precise, we proceed as follows. 
%For a subset $J$ of $\{0,\ldots,N-1\}$, we put
%\begin{equation}
%\mathcal{A}_N(J):=\big\{(p,\alpha)\in\mathcal{A}_N:\,\alpha>a_{N-j,N}(p)
%\,\,\text{for}\,\,j\in J,\,\,\text{and}\,\,\alpha\le a_{N-j,N}(p)
%\,\,\text{for}\,\,j\in J^c\big\},
%\label{eq-entcells}
%\end{equation}
%where we write $J^c:=\{0,\ldots,N-1\}\setminus J$ for the complement.
For a point $(p,\alpha)\in\mathcal{A}_N$, we put
\[
J(p,\alpha):=\big\{j\in\{0,\cdots,N-1\}:\,\,\alpha>a_{N-j,N}(p)\big\},
\]
which defines a function from $\mathcal{A}_N$ to the collection of all
subsets of $\{0,\cdots,N-1\}$.
The 
%Some of the sets $\mathcal{A}_N(J)$ will be empty; we discard these and call
{\em admissible cells} are the level sets (i.e., the sets of constancy) for 
this function $J$. 
As mentioned above, the admissible cells contained in the entangled region 
are called {\em entangled cells}, while those contained in the unentangled 
region are called {\em unentangled cells}; all admissible cells belong to one 
of these categories. 
%{\color{red}
We draw the cell decomposition for $N=2,3$ in 
Figures \ref{fig-1} and \ref{fig-2}.
%} 

We can now improve upon Theorem \ref{thm-entamain2}. 

\begin{thm}
Suppose $(p,\alpha)\in\mathcal{A}_N$. 
%$J$ is a subset of the integer interval $\{0,\ldots,N-1\}$, such that
%$\mathcal{A}_N(J)\ne\emptyset$ is an admissible cell.
%If $(p,\alpha)$ is in $\mathcal{A}_N(J)$, 
Then every $u\in\mathrm{PH}^p_{N,\alpha}(\D)$ has a unique decomposition
\[
u=\sum_{j\in J(p,\alpha)}\Mop^j[w_j],
\]
where each term $\Mop^j[w_j]$ is in $\mathrm{PH}^p_{N,\alpha}(\D)$, while 
the functions $w_j$ are $(N-j)$-harmonic and solve the partial differential
equation $\Lop_{N-j-1}[w_j]=0$ on $\D$, for $j\in J(p,\alpha)$. 
%In addition, for a given $j$,
%we have that $w_j=0$ unless $\alpha>b_{N-j,N}(p)$. 
\label{thm-entamain3}
\end{thm}
 
\begin{rem}
{\rm(i)}
The main improvement in Theorem \ref{thm-entamain3} is that we may now specify 
which terms of the cellular decomposition must necessarily vanish.
The theorem is sharp, in the sense that each term $\Mop^j[w_j]$ with $j\in
J(p,\alpha)$ is allowed to be nontrivial. 
This is easy to see easily by considering Dirac point masses $f$ or uniform 
density (constant) $f=1$  in Proposition \ref{prop-olof1} (compare with 
Lemma \ref{lem-6.2} below). For details, see Subsection \ref{subsec-triv}. 

\noindent{(ii)} The cellular decomposition can be likened to the decomposition
of a vector with respect to a given basis. Theorem \ref{thm-entamain3} tells
us which pieces of the ``basis'' are being used in a given cell.
The number of elements of the set $J(p,\alpha)\subset\{0,\ldots,N-1\}$ 
is like the ``dimension'' of the subspace spanned by the given vectors. 
Perhaps ``degrees of freedom'' would be a better term.   

\noindent{(iii)} \emph{An example of entanglement.} Let us try to explain
what Theorem \ref{thm-entamain3} says in a simple case.
Let $u$ be a 
%real-valued 
biharmonic function in $\D$, so that $N=2$. By the alternative Almansi 
representation \eqref{eq-Almansi2}, $u$ has a unique representation of the form
\[
u(z)=v_0(z)+\Mop[v_1](z)=v_0(z)+(1-|z|^2)v_1(z),
\]
where $v_0,v_1$ are harmonic. 
%permits us to find two holomorphic functions $f_1,f_2$ in the disk $\D$ with
%$\im f_1(0)=\im f_2(0)=0$, such that 
%\[
%u(z)=\re f_1(z)+|z|^2\re f_2(z),\qquad z\in\D.
%\]
Suppose that $u\in\mathrm{PH}^p_{2,\alpha}(\D)$, where
\[
0<p<\tfrac 13,\quad\text{and}\quad-1-3p<\alpha\le -1-2p,
\] 
so that we are in the entangled region (see Figure \ref{fig-1}). 
Theorem \ref{thm-entamain3} then says that $\Lop_1[u]=0$, that is,
\[
0=\Lop_1[u]=\Lop_1[v_0]+\Lop_1\Mop[v_1]=
\Lop_1[v_0]+\Mop\Lop_0[v_1]-8v_1
=4(z\partial_z v_0+\bar z\bar\partial_z v_0)-4v_0-8v_1,
\]
which is equivalent to
\begin{equation}
v_1=\frac12(z\partial_z v_0+\bar z\bar\partial_z v_0)-\frac12 v_0.
\label{xdc}
\end{equation}
In the above calculation, we used the operator identity \eqref{eq-operiden1}.
We conclude that $v_1$ is completely determined by $v_0$. 
%{\color{red}
As the harmonic function $v_0$ is given in terms of its (distributional) 
boundary ``values'' on $\Te$, the relation \eqref{xdc} expresses how 
$u=v_0+\Mop[v_1]$ is given in terms of its distributional boundary ``values'',
which coincide with those of $v_0$. This is of course in perfect agreement with
the Poisson-type representation of Olofsson (see Proposition \ref{prop-olof1}).
For general $N$, related explicit relations hold in every entangled cell; 
these relations involve a certain finite collection of ``free'' harmonic 
functions (like $v_0$ in \eqref{xdc}). 
%}
\end{rem}

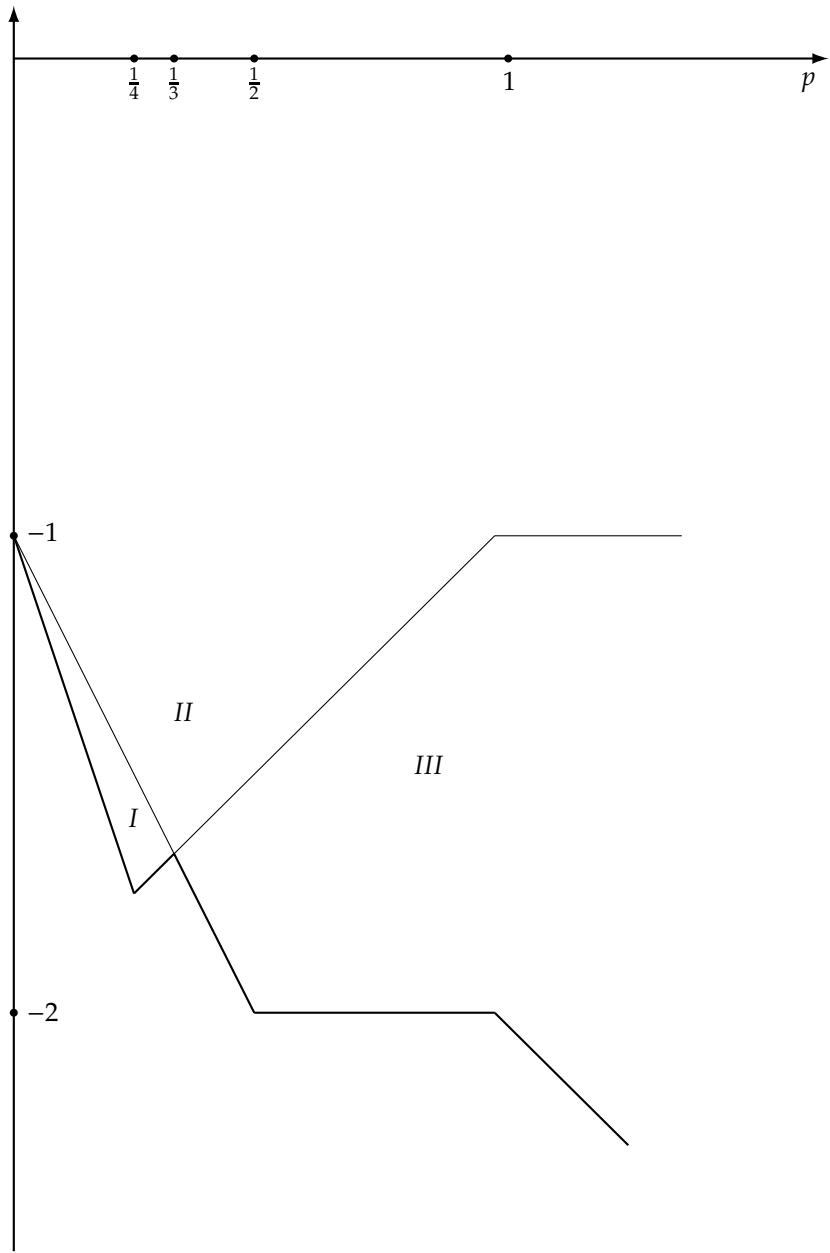
\begin{figure}[p]
\begin{picture}(320,525)
%\caption{XX}
%\linethickness{2pt}
\thicklines
\setlength{\labelwidth}{2em}
\put(0,500){\vector(1,0){305}}
\put(0,50){\vector(0,1){470}}
\put(295,490){$p$}

%\put(180,10){$\beta(3,p)$}
%\put(220,155){$\beta(2,p)$}
\put(185,500){\circle*{3}}
\put(183,488){$1$}
\put(90,500){\circle*{3}}
\put(87,488){$\frac12$}
\put(60,500){\circle*{3}}
\put(57,488){$\frac13$}
\put(45,500){\circle*{3}}
\put(42,488){$\frac14$}
%\put(36,500){\circle*{3}}
%\put(33,488){$\frac15$}
%\put(30,500){\circle*{3}}
%\put(27,488){$\frac16$}
\put(0,320){\circle*{3}}
\put(5,318){$-1$}
\put(0,140){\circle*{3}}
\put(5,137){$-2$}
\thicklines
\put(0,320){\line(1,-3){45}}
%\put(45,280){\circle*{2}}
%\put(53,256){\circle*{2}}
%\put(65,220){\circle*{2}}
\put(45,185){\line(1,1){15}}
%\put(60,200){\circle*{2}}
\put(60,200){\line(1,-2){30}}
%\put(125,160){\circle*{2}}
\put(90,140){\line(1,0){90}}
%\put(245,160){\circle*{2}}
\put(180,140){\line(1,-1){50}}
%\put(5,-100){\circle*{3}}
\thinlines
%\put(45,185){\line(1,-3){45}}
\put(0,320){\line(1,-2){60}}
\put(60,200){\line(1,1){120}}
\put(180,320){\line(1,0){70}}
\put(43,210){$I$}
\put(60,250){$II$}
\put(150,230){$III$}
\end{picture}
\caption{The graph of $\beta(2,p)$ (bottom ``sawtooth'' curve, in thick style), 
plus indication of admissible cells. Here, $I$ is an entangled cell, 
while $II$ and $III$ are unentangled cells ($III$ is principal).}
\label{fig-1}
\end{figure}

\begin{figure}[p]
\begin{picture}(320,565)
%\caption{XX}
\linethickness{2pt}
\thicklines
\put(0,540){\vector(1,0){305}}
\put(0,10){\vector(0,1){550}}
\put(295,530){$p$}

\put(160,540){\circle*{3}}
\put(157,528){$1$}
\put(80,540){\circle*{3}}
\put(77,528){$\frac12$}
\put(53,540){\circle*{3}}
\put(50,528){$\frac13$}
\put(40,540){\circle*{3}}
\put(37,528){$\frac14$}
\put(32,540){\circle*{3}}
\put(29,528){$\frac15$}
\put(26,540){\circle*{3}}
\put(23,528){$\frac16$}
\put(0,380){\circle*{3}}
\put(5,377){$-1$}
\put(0,220){\circle*{3}}
\put(5,217){$-2$}
\put(0,60){\circle*{3}}
\put(5,57){$-3$}

%\put(26,248){\circle*{3}}
%\put(32,252){\circle*{3}}
%\put(40,220){\circle*{3}}
%\put(53,220){\circle*{3}}
%\put(80,140){\circle*{3}}
%\put(160,60){\circle*{3}}

\put(0,380){\line(1,-5){26.5}}
\put(26,247){\line(1,1){5}}
\put(32,252){\line(1,-4){8}}
\put(40,220){\line(1,0){13}}
\put(53,220){\line(1,-3){27}}
\put(80,140){\line(1,-1){80}}
\put(160,60){\line(1,-2){20}}

\thinlines
%\put(45,185){\line(1,-3){45}}
\put(0,380){\line(1,-3){60}}
\put(0,380){\line(1,-4){32}}
\put(31,252){\line(1,1){128}}
\put(160,380){\line(1,0){100}}
\put(53,220){\line(1,0){107}}
\put(160,220){\line(1,-1){70}}
\put(22,270){$I$}
\put(29,265){$II$}
\put(36,237){$III$}
\put(140,180){$IV$}
\put(50,360){$V$}
\put(100,250){$VI$}
\end{picture}
\caption{The graph of $\beta(3,p)$ (bottom ``sawtooth'' curve, in thick style), 
plus indication of admissible cells. Here, $I$, $II$, and $II$ are 
entangled cells, while $IV$, $V$, and $VI$ are unentangled cells 
($IV$ is principal).}
\label{fig-2}
\end{figure}

\section{The polyharmonic Hardy-Littlewood 
ellipticity and its applications}
\label{sec-4}

\subsection{Representation of the extension in terms of the 
Poisson kernel}

If a function 
$f:\D(0,r)\to\C$ is harmonic and extends continuously to the boundary, then
Poisson's formula supplies the representation
\begin{equation}
f(z)=\frac{r^2-|z|^2}{2\pi r}\int_{\Te(0,r)}|z-\zeta|^{-2}f(\zeta)\diff 
s(\zeta),\qquad z\in\D(0,r),
\label{eq-Poisson1}
\end{equation}
where we recall that $\diff s$ stands for arc length measure. 

We suppose that 
$u\in \mathrm{PH}_N(\D)$, i.e. that $u$ is $N$-harmonic in $\D$, and recall 
the notation $\mathbf{E}[u](z,\varrho)$ for the extension
of $u$, given by \eqref{eq-Almansiext}. With $f(z)=\Eop[u](z,\varrho)$
and $r=\varrho$, we obtain from \eqref{eq-Poisson1} that
\begin{equation}
\Eop[u](z,\varrho)=\frac{\varrho^2-|z|^2}{2\pi\varrho}
\int_{\Te(0,\varrho)}|z-\zeta|^{-2}u(\zeta)\diff 
s(\zeta),\qquad z\in\D(0,\varrho),
\label{eq-Poisson2}
\end{equation}
for each $\varrho$ with $0<\varrho<1$, since 
$\Eop[u](\zeta,\varrho)=u(\zeta)$ for $\zeta\in\Te(0,\varrho)$ by 
\eqref{eq-restr1}. In particular, we obtain in an elementary fashion that
\begin{equation}
|\Eop[u](z,\varrho)|\le\frac{1}{2\pi\varrho}\,\frac{\varrho+|z|}{\varrho-|z|}
\int_{\Te(0,\varrho)}|u(\zeta)|\diff s(\zeta),\qquad z\in\D(0,\varrho),
\label{eq-Poisson3}
\end{equation}

\subsection{Lagrangian interpolation of the extension of a 
polyharmonic function}

We suppose that $u\in \mathrm{PH}_N(\D)$, i.e. that $u$ is $N$-harmonic in 
$\D$, and recall the notation $\mathbf{E}[u](z,\varrho)$ for the extension
of $u$, given by \eqref{eq-Almansiext}. 
The polynomial nature of $\varrho\mapsto
\mathbf{E}[u](z,\varrho)$ makes it amenable to Lagrangian interpolation.
Indeed, if we supply real values $\varrho_j$, where $j=1,\ldots,N$, with 
$0<\varrho_1<\varrho_2<\cdots<\varrho_N<1$, then
\begin{equation}
\mathbf{E}[u](z,\varrho)=\sum_{j=1}^N \mathbf{E}[u](z,\varrho_j)\,L_j(\varrho),
\label{eq-Lagint1}
\end{equation}
where $L_j(\varrho)$ is the even interpolation polynomial
\begin{equation}
L_j(\varrho):=\prod_{k:k\ne j}
\frac{\varrho^2-\varrho_k^2}{\varrho_j^2-\varrho_k^2};
\label{eq-Lagintpol1}
\end{equation}
here, it is tacitly assumed that the product runs over the set of integers
$k\in \{1,\ldots,N\}$ (with the exception of $j$).  
In particular, with $\varrho=|z|$, \eqref{eq-restr1} and \eqref{eq-Lagint1}
together give the representation
\begin{equation}
u(z)=\sum_{j=1}^N \mathbf{E}[u](z,\varrho_j)\,L_j(|z|),\qquad z\in\D. 
\label{eq-Lagint2}
\end{equation} 
If we use both \eqref{eq-Poisson2} and \eqref{eq-Lagint2}, we see that
\begin{equation}
u(z)=\sum_{j=1}^N L_j(|z|)\,\frac{\varrho_j^2-|z|^2}{2\pi \varrho_j}
\int_{\Te(0,\varrho_j)}|z-\zeta|^{-2}u(\zeta)\diff 
s(\zeta),\qquad z\in\D(0,\varrho_1). 
\label{eq-Lagint3}
\end{equation} 
We can think about \eqref{eq-Lagint3} as a polyharmonic analogue of the
Poisson representation. Let us introduce the even polynomial
\begin{equation}
M_j(\varrho):=L_j(\varrho)\,\frac{\varrho_j^2-\varrho^2}{2\varrho_j}
=-\frac{\prod_{k}(\varrho^2-\varrho_k^2)}
{2\varrho_j\prod_{k:k\ne j}(\varrho_j^2-\varrho_k^2)},
\label{eq-Mj}
\end{equation}
which is of degree $2n$ and solves the interpolation problem
\[
M_j(\varrho_k)=0 \quad\text{for all}\,\,\, k,\qquad M_j'(\varrho_j)=1.
\]
In terms of the functions $M_j$, the formula \eqref{eq-Lagint2}
simplifies:
\begin{equation}
u(z)=\frac{1}{\pi}\sum_{j=1}^N M_j(|z|)\int_{\Te(0,\varrho_j)}
|z-\zeta|^{-2}u(\zeta)\diff s(\zeta),\qquad z\in\D(0,\varrho_1). 
\label{eq-Lagint4}
\end{equation} 
We may of course apply the gradient to both sides:
\begin{equation}
\nabla u(z)=\frac{1}{\pi}\sum_{j=1}^N \Bigg\{[\nabla M_j(|z|)]
\int_{\Te(0,\varrho_j)}|z-\zeta|^{-2}u(\zeta)\diff s(\zeta)+
M_j(|z|)\int_{\Te(0,\varrho_j)}
[\nabla_z|z-\zeta|^{-2}]u(\zeta)\diff s(\zeta)\Bigg\},
\label{eq-Lagint5}
\end{equation}
for $z\in\D(0,\varrho_1)$. 
%Let us consider the choice of radii
%\[
%\varrho_k^2=\frac14+\frac{k}{2n},\qquad k=1,\ldots,n.
%\]
We write $\varrho_1=\vartheta$, and let $\delta$ denote the quantity
\begin{equation}
\delta:=\min_{j,k:j\ne k}|\varrho_j^2-\varrho_k^2|.
\label{eq-defdelta}
\end{equation}
An elementary calculation then gives that
\begin{equation}
|M_j(|z|)|\le \delta^{-N+1}(\vartheta-|z|)\,\,\,\text{and}\,\,\, 
|\nabla M_j(|z|)|\le N\,\delta^{-N+1} \quad\text{for}\quad z\in\D(0,\vartheta),
\label{eq-elest1}
\end{equation}
while another gives that
\begin{equation}
|z-\zeta|^{-2}\le (\vartheta-|z|)^{-2}\,\,\,\text{and}\,\,\,
|\nabla_z|z-\zeta|^{-2}|\le 2(\vartheta-|z|)^{-3}\quad\text{for}\quad
z\in\D(0,\vartheta),\,\,\,\zeta\in \cup_j\Te(0,\varrho_j).
\label{eq-elest2}
\end{equation}
As we implement these estimates into \eqref{eq-Lagint4} and 
\eqref{eq-Lagint5}, we obtain the following estimates.

\begin{lem} 
If $0<\vartheta=\varrho_1<\cdots<\varrho_N<1$, and $\delta$ is given 
by \eqref{eq-defdelta}, then
\begin{equation*}
|u(z)|\le\frac{\delta^{-N+1}}{(\vartheta-|z|)\pi}
\sum_{j=1}^N \int_{\Te(0,\varrho_j)}
|u(\zeta)|\diff s(\zeta),\qquad z\in\D(0,\vartheta),
%\label{eq-Lagint4-1}
\end{equation*}
and
\begin{equation*}
|\nabla u(z)|\le\frac{(N+2)\delta^{-N+1}}{(\vartheta-|z|)^2\pi}
\sum_{j=1}^N \int_{\Te(0,\varrho_j)}
|u(\zeta)|\diff s(\zeta),\qquad z\in\D(0,\vartheta).
%\label{eq-Lagint4-2}
\end{equation*}
\label{lem-2.1}
\end{lem}

%Let $m_\infty(r)=m_\infty[u](r)$ and $m_1[u](r)$ denote the functions
%\begin{equation}
%m_\infty(r):=\max_{\Te(0,r)}|u|,\quad m_1(r):=\frac{1}{2\pi r}
%\int_{\Te(0,r)}|u|\diff s,
%\label{eq-defsupf}
%\end{equation}
%where $0<r<1$. Moreover, for $p$, $0<p<+\infty$, we introduce the 
%$p$-mean
%\begin{equation}
%m_p(r)=m_p[u](r):=\{m_1[|u|^p](r)\}^{1/p},\quad m_p^p(r)=m_p^p[u](r)
%=m_1[|u|^p](r).
%\label{eq-defpmean}
%\end{equation}
%We see that the first estimate of Lemma \ref{lem-2.1} expresses that
%\begin{equation}
%m_\infty(r)\le\frac{2\delta^{-N+1}}{\vartheta-r}
%\sum_{j=1}^N \varrho_j m_1(\varrho_j),\qquad 0<r<\vartheta.
%\label{eq-Lagint6}
%\end{equation}
%For $0<p<1$, we observe that we have the trivial estimate
%\begin{equation}
%m_1(r)=\frac{1}{2\pi r}\int_{\Te(0,r)}|u|\diff s
%\le m_\infty(r)^{1-p}\int_{\Te(0,r)}|u|^p\frac{\diff s}{2\pi r}
%= m_\infty(r)^{1-p}m_p^p(r),\qquad 0<r<1.
%\label{eq-Lagint7}
%\end{equation}

\subsection{%Local uniform control of point evaluations: 
Hardy-Littlewood
ellipticity}
The next lemma is an elaboration on a theme developed by Hardy and Littlewood 
in 1931, see Theorem 5 of \cite{HL} (see also \cite{Gar}, p. 121). They 
considered $N=1$ (harmonic $u$), and observed first that for $1\le p<+\infty$, 
the function $|u|^p$ is subharmonic, so that
\[
|u(0)|^p\le \frac{1}{\pi}\int_\D|u|^p\diff A.
\]
For $0<p<1$, $|u|^p$ need not be subharmonic. However, Hardy and Littlewood 
found that the above inequality survives nevertheless, if the
right hand side is multiplied by a suitable constant. This is an aspect of
harmonic functions (and of the Laplacian) which we would like to call 
{\em Hardy-Littlewood ellipticity}. 
%{\color{red} 
This fact was generalized to 
harmonic functions in $\mathbb R^n$, with $n>2$, by Fefferman and Stein 
\cite{FeS}. This Hardy-Littlewood ellipticity survives also in the context of 
polyharmonic functions. 

Here is a 1994 result by Pavlovi\'c  \cite{Pa2}, \cite{Pa1} (Lemma 5). 
%}

\begin{lem}
$(0<p<+\infty)$ 
There exists a positive constant $C_1(N,p)$ depending only on $N$ and $p$ 
such that for all $N$-harmonic functions $u$ on the unit disk $\D$ we have
\[
|u(z)|^p\le C_1(N,p)\int_\D|u|^p\diff A,\qquad z\in\D(0,\tfrac12).
\] 
\label{thm-1}
\end{lem}

\begin{proof} 
%{\color{red} 
Here we follow the elegant argument by Pavlovi\'c. 
Alternatively, one can develop an argument based on a selection of optimal 
radii, which shares some features with our previous work \cite{BorHed} 
(see, e.g., the proof of Theorem 8.4 in \cite{HKZ}). 

In the case $1\le p<+\infty$, the asserted estimate can be obtained in a rather
straight-forward fashion based on the integral representation 
\eqref{eq-Lagint3} (sketch: we let the points $\varrho_j$ vary over short
disjoint subintervals of $[\frac34,1[$, and integrate both sides with respect 
to all the $\varrho_j$ over those short intervals. We then apply H\"older's 
inequality).   

We now focus on the remaining case $0<p<1$. 
We quickly realize that it is 
enough to obtain the asserted estimate for a dilate $u_r$ of $u$, where
$u_r(z)=u(rz)$ and $0<r<1$. This allows us to suppose that $u$ extends to be 
$N$-harmonic in a neighborhood of the closed unit disk. In particular, we
may assume that $u$ is bounded in $\D$.  

We pick a $w\in\D$ such that 
\begin{equation}
2|u(w)|^p(1-|w|^2)^2\ge \max_{z\in\D}[(1-|z|^2)^2|u(z)|^p].
\label{eq-pavl1}
\end{equation}
If we write $\rho:=\frac12(1-|w|)$, we quickly realize that \eqref{eq-pavl1}
leads to
\begin{equation}
\max_{\D(w,\varrho)}|u|^p\le 32|u(w)|^p.
\label{eq-pavl2}
\end{equation}
Then, in view of the already established bound for $p=1$ (see the first portion
of this proof), and \eqref{eq-pavl1}, we have the estimate  
\begin{multline*}
|u(w)|(1-|w|^2)^2\le c_1(N)\int_{\D(w,\rho)}|u|\diff A\le 
c_1(N)\max_{\D(w,\rho)}|u|^{1-p}\int_{\D(w,\rho)}|u|^p\diff A
\\
\le 
32^{(1-p)/p}\,c_1(N)|u(w)|^{1-p}\int_{\D}|u|^p\diff A.
\end{multline*}
Here, $c_1(N)$ is a suitable positive contant which only depends on $N$.
It now follows that
\[
|u(z)|^p\le 4\,(1-|z|^2)^2|u(z)|^p\le
8\,|u(w)|^p(1-|w|^2)^2\le 8(32^{(1-p)/p})\, c_1(N)\int_{\D}|u|^p\diff A,
\qquad z\in\D(0,\tfrac12).
\]
as required. 
\end{proof}

As a consequence, we obtain the following result of Pavlovi\'c:

\begin{cor}
$(0<p<+\infty)$ 
There exists a positive constant $C_2(N,p)$ depending only on $N$ and $p$ 
such that for all $N$-harmonic functions $u$ on the unit disk $\D$ we have
\[
|\nabla u(z)|^p\le C_2(N,p)\int_\D|u|^p\diff A,\qquad z\in\D(0,\tfrac12).
\]
\label{cor-1}
\end{cor}

%\begin{proof}
%Theorem \ref{thm-1} and its proof show that $|u|^p$ is uniformly bounded 
%on $\D(0,\frac34)$ in terms of
%\[
%\int_\D|u|^p\diff A. 
%\]
%The desired control of the gradient may now be obtained as a consequence of 
%Lemma \ref{lem-2.1}.
%\end{proof}

%\subsection{Hardy-Littlewood type estimates on general disks}
%Here, we shall apply Theorem \ref{thm-1} and Corollary \ref{cor-1} to the 
%{\color{red}
For future use, we formulate the above results in
%} 
the setting of a general disk $\D(z_0,r)$. 

\begin{cor}
$(0<p<+\infty)$ 
Suppose $u$ is $N$-harmonic in the disk $\D(z_0,r)$, where $z_0\in\C$ and the
radius $r$ is positive. Then there exist positive constants $C_1(N,p)$
and $C_2(N,p)$ depending only on $N$ and $p$ such that
\[
|u(z)|^p\le \frac{C_1(N,p)}{r^2}\int_{\D(z_0,r)}|u|^p\diff A,\qquad z\in
\D(z_0,\tfrac12 r),
\] 
and
\[
|\nabla u(z)|^p\le \frac{C_2(N,p)}{r^{2+p}}\int_{\D(z_0,r)}|u|^p\diff A,
\qquad z\in\D(z_0,\tfrac12 r).
\]
\label{cor-2}
\end{cor}

%\begin{proof}
%The $N$-harmonic functions are unperturbed by the change of variables 
%$z=z_0+r\zeta$ where $\zeta\in\D$ and $z\in\D(z_0,r)$. The asserted estimates 
%are now immediate consequences of Theorem \ref{thm-1} and of Corollary 
%\ref{cor-1}.
%\end{proof}

%\begin{rem}
%There is of course no serious loss of generality if we would formulate the
%above corollary with the special choice $z=z_0$ (the center point of the 
%disk). 
%\end{rem}

%\subsection{Pointwise bounds on the function and on the gradient}
%{\color{red} 
Finally, Pavlovi\'c also obtains effective pointwise and integral 
bounds on $u$ and $\nabla u$ given that $u\in \mathrm{PH}^p_{N,\alpha}(\D)$. 
%}

\begin{cor} $(0<p<+\infty)$ Suppose $u$ is $N$-harmonic in $\D$ and that 
$\alpha$ is a real parameter. Then there exist constants 
$C_3(N,p,\alpha)$ and $C_4(N,p,\alpha)$ which depend only on $N,p,\alpha$, 
such that
\[
|u(z_0)|^p\le \frac{C_3(N,p,\alpha)}{(1-|z_0|)^{\alpha+2}}\int_\D|u(z)|^p
(1-|z|^2)^\alpha\diff A(z),\qquad z_0\in\D,
\]
and 
\[
|\nabla u(z_0)|^p\le \frac{C_4(N,p,\alpha)}{(1-|z_0|)^{\alpha+p+2}}
\int_\D|u(z)|^p(1-|z|^2)^\alpha\diff A(z),\qquad z_0\in\D.
\]
\label{cor-2.6}
\end{cor}

%\begin{proof}
%Pick $z_0\in\D$; without loss of generality, we may assume that $\frac34\le
%|z_0|<1$. We then observe that
%\[
%\frac{3}{4}(1-|z_0|)\le 1-|z|^2\le 3(1-|z_0|),\qquad 
%z\in\D(z_0,\tfrac12(1-|z_0|)),
%\]
%which leads to the estimate
%\begin{equation}
%(1-|z_0|)^\alpha\le 3^{|\alpha|}\,(1-|z|^2)^\alpha,\qquad 
%z\in\D(z_0,\tfrac12(1-|z_0|)).
%\label{eq-estweight1}
%\end{equation}
%From Corollary \ref{cor-2} with $r=\frac12(1-|z_0|)$, we get that
%\begin{equation*}
%|u(z_0)|^p\le \frac{4C_1(N,p)}{(1-|z_0|)^2}\int_{\D(z_0,\frac12(1-|z_0|))}
%|u|^p\diff A,
%\end{equation*}
%which combined with \eqref{eq-estweight1} leads to
%\begin{equation*}
%(1-|z_0|)^{2+\alpha}|u(z_0)|^p\le 
%3^{|\alpha|}4C_1(N,p)\int_{\D(z_0,\frac12(1-|z_0|))}
%|u(z)|^p(1-|z|^2)^\alpha\diff A.
%\end{equation*}
%The first assertion is now immediate with $C_3(N,p,\alpha)=
%3^{|\alpha|}4C_1(N,p)$, since $\D(z_0,\frac12(1-|z_0|))$ is contained in $\D$. 
%The assertion involving the gradient is settled in a similar fashion using the
%second estimate of Corollary \ref{cor-2}.  
%\end{proof}

%\subsection{The integral control of the gradient} 

%We may now control the integral of the gradient in terms of the integral of
%the function. This is important for our analysis of the spaces
%$\mathrm{PH}_{N,\alpha}^p(\D)$. 

%{\color{red}Finally, Pavlovi\'c obtains effective pointwise and integral 
%bounds on $u$ and $\nabla u$ given that $u\in \mathrm{PH}^p_{N,\alpha}(\D)$. }

\begin{cor} $(0<p<+\infty)$
Suppose $u$ is $N$-harmonic in the unit disk $\D$ and that $\alpha$ is a 
real parameter. Then there exists a constant $C_5(N,p,\alpha)$
which depends only on $N,p,\alpha$, such that 
\[
\int_\D|\nabla u|^p(1-|z|^2)^{\alpha+p}\diff A(z)\le C_5(N,p,\alpha)
\int_\D|u|^p(1-|z|^2)^{\alpha}\diff A(z).
\]
\label{thm-2}
\end{cor}

\begin{cor} $(0<p<+\infty)$
If $u\in\mathrm{PH}_{N,\alpha}^p(\D)$, then $\partial_z u$ and 
$\bar\partial_z u$ are both in $\mathrm{PH}_{N,\alpha+p}^p(\D)$.
\label{cor-3}
\end{cor}

%\begin{proof}
%By \eqref{eq-nabla2}, the assertion of the corollary is immediate from 
%Corollary \ref{thm-2}.
%\end{proof}

\begin{cor} $(0<p<+\infty)$
If $u\in\mathrm{PH}_{N,\alpha}^p(\D)$, then $\partial_z^j \bar\partial_z^k 
u\in \mathrm{PH}_{N,\alpha+(j+k)p}^p(\D)$ for $j,k=0,1,2,\ldots$. 
\label{cor-4}
\end{cor}

%\begin{proof}
%The assertion follows by iteration of Corollary \ref{cor-3}.
%\end{proof}

\subsection{Control of the antiderivative}

A particular instance of Corollary \ref{cor-3} is when $f$ is holomorphic in
$\D$: If $f\in L^p_{\alpha}(\D)$ then $f'\in L^p_{\alpha+p}(\D)$. In the
converse direction, we have the following. 

\begin{prop}
$(0<p<+\infty)$ Suppose $f$ is holomorphic in $\D$, with 
$f'\in L^p_{\alpha+p}(\D)$ for some real parameter $\alpha$. 
Then: 

\noindent{\rm(a)} If $\alpha\le-1-p$, then $f$ is constant.

\noindent{\rm(b)} If $\alpha>-1$ and $1\le p<+\infty$, then
$f\in L^p_{\alpha}(\D)$. 

\noindent{\rm(c)} If $-1-p<\alpha\le-1$ and $1\le p<+\infty$, then
$f\in L^p_{\beta}(\D)$ for every $\beta>-1$. 

\noindent{\rm(d)} If $\frac12<p<1$ and $-1-p<\alpha\le p-2$, then 
$f\in L^1_{\beta}(\D)\cap L^p_{\beta}(\D)$ for every $\beta>-1$.

\noindent{\rm(e)} If $0<p<1$ and $\alpha>p-2$, then 
$f\in L^1_{-2+(\alpha+2)/p}(\D)$. Moreover, the inclusion
$L^1_{-2+(\alpha+2)/p}(\D)\subset L^p_{\alpha+1-p+\epsilon}(\D)$ holds 
for every $\epsilon>0$.
\label{prop-2.10}
\end{prop}

\begin{proof}
If $\alpha\le-1-p$ we must have $f'(z)\equiv0$ which makes $f$ constant. This
settles part (a). As for parts (b) and (c), this follows from Proposition 
1.11 of \cite{HKZ}. 

We turn to the case $0<p<1$. Holomorphic functions are harmonic, so 
Corollary \ref{cor-2.6} applied to $u=f'$ and $N=1$ tells us that the function
\[
(1-|z|^2)^{\alpha+p+2}|f'(z)|^p
\]
is uniformly bounded in the disk $\D$. It now follows from this and
our assumption $f'\in L^p_{\alpha+p}(\D)$ that
\[
\int_\D|f(z)|(1-|z|^2)^{\alpha+p+(\alpha+p+2)(1-p)/p}\diff A(z)<+\infty,
\] 
which we simplify to $f'\in L^1_{-1+(\alpha+2)/p}(\D)$. If $\alpha\le p-2$,
this gives that $f\in L^1_{\beta}(\D)$ for every $\beta>-1$, by
Proposition 1.11 of \cite{HKZ}, which settles part (d), if we also use 
H\"older's inequality to obtain the $L^p$ statement. If instead $\alpha
>p-2$, then Proposition 1.11 of \cite{HKZ} tells us that 
$f\in L^1_{-2+(\alpha+2)/p}(\D)$, and part (e) follows, except for the 
inclusion. The inclusion is a simple consequence of H\"older's inequality.
\end{proof}

\subsection{Control of individual terms in the Almansi expansion}
With the help of Proposition \ref{prop-2.10}, we may now obtain integral
control of the individual terms of the Almansi expansion.
 
\begin{cor} $(0<p<+\infty)$
If $u\in\mathrm{PH}_{N,\alpha}^p(\D)$ has the Almansi expansion 
\[
u(z)=u_0(z)+|z|^2 u_1(z)+\cdots+|z|^{2N-2}u_{N-1}(z),
\]
where the $u_j$ are harmonic in $\D$, then

\noindent{\rm(a)} if $1\le p<+\infty$ and $\alpha>-1-(N-1)p$, then
$u_{N-1}\in L^p_{\alpha+(N-1)p}(\D)$, and

\noindent{\rm(b)} if $0<p<1$ and $\alpha\ge-1-N+p$, then
$u_{N-1}\in L^p_{\alpha+N-p+\epsilon}(\D)$, for every $\epsilon>0$.

\noindent{\rm(c)} if $1\le p<+\infty$ and $\alpha\le-1-(N-1)p$, or if 
$0<p<1$ and $\alpha\le-1-N+p$, then
$u_{N-1}\in L^p_{-1+\epsilon}(\D)$, for every $\epsilon>0$.

\label{cor-5}
\end{cor}

\begin{proof}
We split $u_j=f_j+\overline{g_j}$, where $f_j,g_j$ are holomorphic, with 
$g_j(0)=0$. 
We calculate that 
\[
\partial_z\Delta^{N-1}u(z)=(N-1)!4^{N-1}\partial_z^{N}[z^{N-1}f_{N-1}(z)],
\]
and see from Corollary \ref{cor-4} that $\partial_z\Delta^{N-1}u$ is a 
holomorphic function in $L^p_{\alpha+(2N-1)p}(\D)$. Integrating backwards 
using Proposition \ref{prop-2.10} shows that if $1\le p<+\infty$ and
%{\color{red}
$\alpha>-1-(N-1)p$,
%}, 
then $f_{N-1}\in L^p_{\alpha+(N-1)p}(\D)$, while if 
$0<p<1$ and $\alpha\ge-1-N+p$, we instead have
$f_{N-1}\in L^p_{\alpha+N-p+\epsilon}(\D)$ for every $\epsilon>0$. If instead
$1\le p<+\infty$ and $\alpha\le-1-(N-1)p$ or $0<p<1$ and $\alpha\le-1-N+p$,
we find that $f_{N-1}\in L^{p}_{-1+\epsilon}(\D)$ for every $\epsilon>0$.
Analogously,
the function $g_{N-1}$ has the same integrability properties, and then 
$u_{N-1}=f_{N-1}+\bar g_{N-1}$ automatically has the asserted properties.
The proof is complete.
\end{proof}

\subsection{Divisibility of a polyharmonic functions by a canonical 
factor}
Some $N$-harmonic functions $u(z)$ are of the form 
%{\color{red}
$(1-|z|^2)\widetilde{u}(z)$,
%}, 
where $\widetilde{u}$ is $(N-1)$-harmonic. The following result offers 
an instance when this happens.

\begin{prop} $(0<p<+\infty)$ Suppose $u\in\mathrm{PH}^p_{N,\alpha}(\D)$ for some
integer $N=1,2,3,\ldots$, and a real $\alpha$ with 
$\alpha\le\min\{p-2,-1\}$. Then if $N\ge2$, $u$ has the form
$u(z)=(1-|z|^2)\widetilde u(z)$, where $\widetilde u\in
\mathrm{PH}^p_{N-1,\alpha+p}(\D)$, while if $N=1$, we have that $u(z)\equiv0$.
%{\rm(i)} $0<p\le1$ and 
%$\alpha\le p-2$, or if {\rm(ii)} $1\le p<+\infty$ and $\alpha\le -1$.
\label{prop-5.5}
\end{prop}

\begin{proof}
%\textsc{Step 1}: 
We first consider the case $0<p<1$. Since then $\alpha\le p-2$, we have that
$u\in\mathrm{PH}_{N,\alpha}(\D)\subset\mathrm{PH}_{N,p-2}(\D)$, because
\begin{equation}
\|u\|^p_{p,p-2}=\int_\D |u(z)|^p(1-|z|^2)^{p-2}\diff A(z)\le
\int_\D |u(z)|^p(1-|z|^2)^{\alpha}\diff A(z)=\|u\|^p_{p,\alpha}<+\infty,
\label{eq-est3.11-1}
\end{equation}
and the pointwise estimate of Corollary \ref{cor-2.6} tells us that
\begin{equation}
[(1-|z|^2)|u(z)|]^p\le 2^p C_3(N,p,p-2)\,\|u\|^p_{p,p-2},\qquad z\in\D.
\label{eq-est3.12-1}
\end{equation}
We conclude from \eqref{eq-est3.11-1} and \eqref{eq-est3.12-1} that
\begin{equation}
\int_\D \frac{|u(z)|}{1-|z|^2}\diff A(z)\le
2^{1-p}[C_3(N_0,p,p-2)]^{(1-p)/p}\|u\|_{p,\alpha}<+\infty,
\label{eq-est3.13-1}
\end{equation}
which by an elementary argument involving polar coordinates implies that 
\begin{equation}
\liminf_{\varrho\to 1^-}\int_{\Te(0,\varrho)} |u|\diff s=0.
\label{eq-3.14-1}
\end{equation}
From the alternative Almansi representation \eqref{eq-Almansi2}, we have that
\[
u(z)=v_0(z)+(1-|z|^2)v_1(z)+\cdots+(1-|z|^2)^{N-1}v_{N-1}(z),
\]
where the functions $v_j$ are all harmonic in $\D$. The extension of $u$
can then be written as
\[
\mathbf{E}[u](z,\varrho)=v_0(z)+(1-\varrho^2)v_1(z)+\cdots+
(1-\varrho^2)^{N-1}v_{N-1}(z),
\]
and by \eqref{eq-Poisson3} combined with \eqref{eq-3.14-1}, we find that
\begin{equation}
|v_0(z)|=\lim_{\varrho\to1^-}
|\Eop[u](z,\varrho)|\le\liminf_{\varrho\to1^-}
\frac{1}{2\pi\varrho}\,\frac{\varrho+|z|}{\varrho-|z|}
\int_{\Te(0,\varrho)}|u(\zeta)|\diff s(\zeta)=0,\qquad z\in\D,
\label{eq-Poisson3.1-1}
\end{equation} 
that is, $v_0(z)\equiv0$. In case $N=1$, we are finished. In case $N\ge2$,
we note that by the alternative Almansi representation, this means that the 
function
\[
\widetilde u(z):=
\frac{u(z)}{1-|z|^2}=v_1(z)+(1-|z|^2)v_2(z)\cdots+(1-|z|^2)^{N-2}v_{N-1}(z),
\]
is $(N-1)$-harmonic. 
It follows that $\widetilde u\in\mathrm{PH}^p_{N-1,\alpha+p}(\D)$, as claimed. 

We finally turn to the case $1\le p<+\infty$.
Since $\alpha\le-1$, we have
\[
\|u\|_{p,-1-p}^p=\int_{\D}\frac{|u(z)|^p}{1-|z|^2}\diff A(z)\le
\int_{\D}|u(z)|^p(1-|z|^2)^{\alpha}\diff A(z)<+\infty.
\]
An elementary argument now shows that
\[
\liminf_{\varrho\to1^-}\int_{\Te(0,\varrho)}|u|^p\diff s
=0
\]
so that in particular (recall that $1\le p<+\infty$) -- 
by H\"older's inequality --
\[
\liminf_{\varrho\to1^-}\int_{\Te(0,\varrho)}|u|\diff s
=0.
\]
This is \eqref{eq-3.14-1}. We may then use \eqref{eq-Poisson3.1-1} to conclude
that $v_0(z)\equiv0$, and if $N=1$, we are finished. If $N\ge2$, we obtain
instead that $u(z)=(1-|z|^2)\widetilde u(z)$ where $\widetilde u$ is 
$(N-1)$-harmonic. Then $\widetilde u\in\mathrm{PH}^p_{N-1,\alpha+p}(\D)$, as 
claimed.
\end{proof}

\subsection{A criterion for the triviality of a polyharmonic function}
We obtain a sufficient criterion for $\mathrm{PH}^p_{N,\alpha}(\D)=\{0\}$.  

\begin{prop} $(0<p<+\infty)$ 
Suppose $\alpha$ is real with $\alpha\le-1-(2N-1)p$,  for some
integer $N=1,2,3,\ldots$. Then $\mathrm{PH}^p_{N,\alpha}(\D)=\{0\}$. 
\label{prop-5.6}
\end{prop}

\begin{proof}
As the spaces $\mathrm{PH}^p_{N,\alpha}(\D)$ grow with $\alpha$, it is enough
to obtain the result when $\alpha$ is critically big: $\alpha=-1-(2N-1)p$.

\medskip

\noindent\textsc{Step 1}. {\em We show that the assertion holds for $N=1$}:
$\mathrm{PH}^p_{1,-1-p}(\D)=\{0\}$. To this end,
we pick a function 
%{\color{red}
$v\in\mathrm{PH}^p_{1,-1-p}(\D)$,
%}
and observe that since $v$
is harmonic, $\partial_zv$ is holomorphic while $\bar\partial_zv$ is 
conjugate-holomorphic. Moreover, Corollary \ref{cor-3} tells us that
$\partial_z v,\bar\partial_z v\in\mathrm{PH}^p_{1,-1}(\D)$. Now 
$|\partial_z v|^p$ and $|\bar\partial_zv|^p$ are both subharmonic in $\D$, 
and in particular, their means on the circles $|z|=r$ increase with $r$, 
$0<r<1$. Using polar coordinates, then, we realize that
\[
\|\partial_zv\|^p_{p,-1}+\|\bar\partial_zv\|^p_{p,-1}=
\int_\D\big(|\partial_z v(z)|^p+|\bar\partial_z v(z)|^p\big)
\frac{\diff A(z)}{1-|z|^2}<+\infty
\]
forces $\partial_zv=0$ and $\bar\partial_zv=0$, so that $v$ must be
constant. As the only constant function in $\mathrm{PH}^p_{1,-1-p}(\D)$ is the
zero function, we obtain $v=0$, and the conclusion 
$\mathrm{PH}^p_{1,-1-p}(\D)=\{0\}$ is immediate.

\medskip

\noindent\textsc{Step 2}. {\em We show that the assertion holds for $N>1$}:
$\mathrm{PH}^p_{1,-1-(2N-1)p}(\D)=\{0\}$. We pick a function 
$u\in\mathrm{PH}^p_{N,-1-(2N-1)p}(\D)$ and intend to obtain that $u=0$. Since 
$u$ is $N$-harmonic, the function $v_1:=\Delta^{N-1}u$ is harmonic. Moreover, 
as 
\[
v_1(z)=\Delta^{N-1}u(z)=4^{N-1}\partial_z^{N-1}\bar\partial_z^{N-1}u(z)
\]
Corollary \ref{cor-4} now tells us that 
$v_1=\Delta^{N-1}u\in\mathrm{PH}^p_{1,-1-p}(\D)$. This case we handled in Step 1,
so that we know that $v_1=0$. But then the function  $v_2:=\Delta^{N-2}u$ is
harmonic, and by Corollary \ref{cor-4}, 
$v_2=\Delta^{N-2}u\in\mathrm{PH}^p_{1,-1-3p}(\D)=\{0\}$ because trivially
$\mathrm{PH}^p_{1,-1-3p}(\D)\subset\mathrm{PH}^p_{1,-1-p}(\D)=\{0\}$. If $N=2$,
we have arrived at the conclusion that $u=0$, as needed. If $N>2$, we continue,
and form the successively the functions  $v_j:=\Delta^{N-j}u$ which belong to
$\mathrm{PH}^p_{1,-1-(2j-1)p}(\D)=\{0\}$ for $j=3,\ldots,N$. With $j=N$ we
arrive at $u=v_N=0$, as needed. 
\end{proof}

\section{The weighted integrability structure of polyharmonic 
functions}
\label{sec-6}

\subsection{A family of polyharmonic kernels}

For $N=1,2,3,\ldots$ and $j=0,1,\ldots,N$, let 
\begin{equation}
U_{j,N}(z):=\frac{(1-|z|^2)^{N+j-1}}{|1-z|^{2j}},
\label{eq-defUNk}
\end{equation}
so that with 
\[
U_{j,j}(z)=\frac{(1-|z|^2)^{2j-1}}{|1-z|^{2j}},
\]
we obtain the relation
\begin{equation}
U_{j,N}(z)=(1-|z|^2)^{N-j}U_{j,j}(z).
\label{eq-relUNk}
\end{equation}
The functions $U_{j,N}$ are $N$-harmonic, and as we shall see, they are 
extremal for the critical integrability type $\beta(N,p)$. The function
$U_{1,1}$ is the Poisson kernel for the boundary point at $1$, and
the function $U_{1,2}$ is known in the context of the bilaplacian as the 
{\em harmonic compensator} \cite{Hed1}. We also note that the function 
$U_{2,2}$ appears implicitly in the biharmonic setting in \cite{AH} 
($U_{2,2}(z)=2F(z,1)-H(z,1)$ in their notation); cf. \cite{Olof1}. 
More recently, in \cite{Olof2} the kernels $U_{N,N}$ are shown to solve the 
Dirichlet problem in the disk $\D$ for a certain (singular) second order
elliptic differential operator. It remains to substantiate that $U_{j,N}$ is 
$N$-harmonic. 

\begin{lem}
For $N=1,2,3,\ldots$ and $j=0,1,\ldots,N$, the functions $U_{j,N}$ are all
$N$-harmonic in $\C\setminus\{1\}$.
\label{lem-6.1}
\end{lem}

\begin{proof}
By inspection, the function $U_{0,N}(z)=(1-|z|^2)^{N-1}$ is $N$-harmonic.
As for the other kernels $U_{j,N}$ with $1\le j\le N$, the identity 
\eqref{eq-relUNk} together with the Almansi representation \eqref{eq-Almansi1}
shows that it is enough to know that $U_{j,j}$ is $j$-harmonic. 
Flipping variables, we just need to show that $U_{N,N}$ is $N$-harmonic. 
To this end, we change variables to $\zeta=1-z$ and see from the binomial 
theorem that
%{\color{red}
\begin{gather*}
U_{N,N}(1-\zeta)=\frac{(\zeta+\bar\zeta
-\zeta\bar\zeta)^{2N-1}}{\zeta^N\bar\zeta^N}
=\sum_{j,k}(-1)^{j+k+1}\frac{(2N-1)!}{j!k!(2N-j-k-1)!}
\zeta^{N-k-1}\bar\zeta^{N-j-1},
\end{gather*}
%}
where $j,k$ range over integers with $j,k\ge0$ and $j+k\le 2N-1$. 
It follows that $U_{N,N}$ is $N$-harmonic for $\zeta\ne0$, because in every 
term of the above sum, we have either $0\le N-k-1\le N-1$ or 
%{\color{red}
$0\le N-j-1\le N-1$,
%} 
or both. An application of $\Delta_\zeta^N=
4^N\partial_\zeta^N\bar\partial_\zeta^N$ to each term of the finite sum
then then results in $0$, as needed. 
\end{proof}

%For $n\ge 1$ and $k\ge 0$ we define 
%\begin{gather}
%u_n(z)=\frac{(1-|z|^2)^{2n-1}}{|1-z|^{2n}},\label{vo1}\\
%u_{n,k}(z)=(1-|z|^2)^k u_n(z).\notag
%\end{gather}
%
%We have $u_n\in\mathcal H_n$ and, hence (by Lemma 1), 
%$u_{n,k}\in\mathcal H_{n+k}$. Indeed, passing to $w=1-z$ we get
%\begin{gather*}
%\Delta^n u_n=\Delta^n\frac{(w+\bar w-w\bar w)^{2n-1}}{w^n\bar w^n}\\=
%\Delta^n\sum_{k_1,k_2\ge 0,\,k_1+k_2\le 2n-1}
%\frac{(2n-1)!}{k_1!k_2!(2n-1-k_1-k_2)!}
%(-1)^{2n-1-k_1-k_2}\times\\
%\times w^{k_1+2n-1-k_1-k_2-n}\bar w^{k_2+2n-1-k_1-k_2-n}=0,
%\end{gather*}
%because in every term of our sum either $0\le n-1-k_2\le n-1$ or 
%$0\le n-1-k_1\le n-1$.
\medskip

We recall the definition of the functions $b_{j,N}(p)$ from \eqref{eq-bj}
and \eqref{eq-b0}.

\begin{lem}
$(0<p<+\infty)$ For $N=1,2,3\ldots$, $j=0,\ldots,N$, and real $\alpha$,
we have that
\[
U_{j,N}\in\mathrm{PH}^p_{N,\alpha}(\D)
\,\,\,\Longleftrightarrow\,\,\, \alpha>b_{j,N}(p).
\]
\label{lem-6.2}
\end{lem}

\begin{proof}
To decide when $U_{j,N}\in\mathrm{PH}^p_{N,\alpha}(\D)$, we calculate:
\begin{equation}
\|U_{j,N}\|_{p,\alpha}^p=\int_{\D}|U_{j,N}|^p(1-|z|^2)^\alpha\diff A=
\int_\D\frac{(1-|z|^2)^{(N+j-1)p+\alpha}}{|1-z|^{2jp}}\diff A(z).
\end{equation}
By Lemma \ref{lem-int}, we have the following. For $j=0$, 
$U_{j,N}=U_{0,N}\in\mathrm{PH}^p_{N,\alpha}(\D)$ if and only if  
%{\color{red}
$(N-1)p+\alpha>-1$.
%}
For $1\le j\le N$, however, 
$U_{j,N}\in\mathrm{PH}^p_{N,\alpha}(\D)$ if and only if  
both $(N+j-1)p+\alpha>-1$ and $(N+j-1)p+\alpha>2(jp-1)$.
After some trivial algebraic manipulations, the assertion of the lemma is 
now immediate.
\end{proof}

So, if $\alpha$ meets
\[
\alpha>\min_{j:0\le j\le N}b_{j,N}(p),
\]
then one of the functions $U_{0,N},U_{1,N},\ldots,U_{N,N}$ will be in
$\mathrm{PH}^p_{N,\alpha}(\D)$, so that in particular, 
\[
\mathrm{PH}^p_{N,\alpha}(\D)\ne\{0\}.
\]
It is quite remarkable that this criterion is also necessary for 
$\mathrm{PH}^p_{N,\alpha}(\D)\ne\{0\}$, as Theorem \ref{thm-main2} says.
The proof will be supplied in Section \ref{sec-applcellul}.

%\begin{thm}
%$(0<p<+\infty)$ For $N=1,2,3,\ldots$ and real $\alpha$, 
%$\mathrm{PH}^p_{N,\alpha}(\D)\ne\{0\}$ holds if and only if 
%\[
%\alpha>\min_{j:0\le j\le N}b_{j,N}(p),
%\]
%where the functions $b_{j,N}(p)$ are given by \eqref{eq-bj} and \eqref{eq-b0}.
%\label{thm-main3}
%\end{thm}

%\begin{proof}[Proof of Theorem \ref{thm-main2}]
%The implication ``$\implies$'' part is treated above. So we turn to the 
%remaining ``$\Longleftarrow$'' part. 
%\medskip

\section{The structure of polyharmonic functions: the cellular
decomposition}
\label{sec-7}

\subsection{The basic properties of the partial differential operators 
$\Lop_\theta$}
We recall that $\Lop_\theta$ is the partial differential operator given by
\eqref{eq-LN-0}, 
\[
\Lop_{\theta}[u](z)=
(1-|z|^2)\Delta u(z)+4\theta[z\partial_z u(z)+\bar z\bar\partial_z u(z)]
-4\theta^2u(z),
\]
and that $\Mop$ is the operator of multiplication by $1-|z|^2$.
The basic operator identities satisfied by $\Lop_\theta$ are the following:
\begin{equation}
\Delta \Lop_{\theta}=\Lop_{\theta-1}\Delta,
\label{eq-operiden0}
\end{equation}
and
\begin{equation}
\Lop_{\theta}\Mop=\Mop\Lop_{\theta-1}-8\theta \Iop,
\label{eq-operiden1}
\end{equation}
where $\Iop$ is the identity operator.
To obtain \eqref{eq-operiden0}, we calculate as follows:
\begin{multline*}
\Delta\Lop_\theta[u](z)=\Delta\big[(1-|z|^2)\Delta u(z)+
4\theta[z\partial_zu(z)+\bar z\bar\partial_zu(z)]
-4\theta^2u(z)\big]
\\
=(1-|z|^2)\Delta ^2u(z)-4z\partial_z \Delta u(z)-
4\bar z\bar\partial_z\Delta u(z)-4\Delta u(z)
\\
+4\theta[2\Delta u(z)+z\partial_z\Delta u(z)+\bar z\bar\partial_z\Delta u(z)]
-4\theta^2\Delta u(z)
\\
%{\color{red}
=(1-|z|^2)\Delta^2 u(z)+4(\theta-1)[z\partial_z\Delta u(z)
+\bar z\bar\partial_z\Delta u(z)]-4(\theta-1)^2\Delta u(z) 
=\Lop_{\theta-1}\Delta u(z).
%}
\end{multline*}
To instead obtain \eqref{eq-operiden1}, we calculate somewhat analogously:
\begin{multline*}
\Lop_\theta\Mop[u](z)=(1-|z|^2)\Delta[(1-|z|^2)u(z)]+
4\theta[z\partial_z((1-|z|^2)u(z))+\bar z\bar\partial_z((1-|z|^2)u(z))]
\\
-4\theta^2(1-|z|^2)u(z)=(1-|z|^2)[(1-|z|^2)\Delta u(z)-4z\partial_z u(z)
-4\bar z\bar\partial_z u(z)-4u(z)]
\\
+4\theta[-2|z|^2u(z)+(1-|z|^2)(z\partial_zu(z)+\bar z\bar\partial_zu(z))]
-4\theta^2(1-|z|^2)u(z)
\\
=(1-|z|^2)^2\Delta u(z)+4(\theta-1)(1-|z|^2)[z\partial_z u(z)
+\bar z\bar\partial_z u(z)]-4(\theta-1)^2(1-|z|^2)u(z) 
\\
-8\theta u(z)
=(1-|z|^2)\Lop_{\theta-1}[u](z)-8\theta u(z).
\end{multline*}
By iteration of the operator identity \eqref{eq-operiden0}, we obtain more
generally that
\begin{equation}
\Delta^j \Lop_{\theta}=\Lop_{\theta-j}\Delta^j,\qquad j=0,1,2,\ldots.
\label{eq-operiden0'}
\end{equation}
If we instead iterate \eqref{eq-operiden1}, we find 
%{\color{red}
the following
%} 
operator identity:
\begin{equation}
\Lop_\theta\Mop^j=\Mop^j\Lop_{\theta-j}+4j(j-1-2\theta)\Mop^{j-1},
\qquad j=0,1,2,\ldots.
\label{eq-operiden2}
\end{equation}

\begin{prop}
We have the following factorization:
\[
\Lop_0\Lop_1\cdots\Lop_{n-1}=\Mop^n\Delta^n,\qquad n=1,2,3,\ldots.
\]
\label{prop-basicoe1}
\end{prop}

\begin{proof}
We argue by induction. Since $\Lop_0=\Mop\Delta$, the assertion holds 
trivially for $n=1$.
Next, suppose we have established the identity for $n=n_0\ge1$, that is,
\[
\Lop_0\Lop_1\cdots\Lop_{n_0-1}=\Mop^{n_0}\Delta^{n_0}
\]
holds. We then see that
\[
\Lop_0\Lop_1\cdots\Lop_{n_0}=\Mop^{n_0}\Delta^{n_0}\Lop_{n_0}
=\Mop^{n_0}\Lop_0\Delta^{n_0}=\Mop^{n_0+1}\Delta^{n_0+1},
\]
where we used the iterated operator identity \eqref{eq-operiden0'} and that
$\Lop_0=\Mop\Delta$. The proof is complete.
\end{proof}

\begin{cor}
Fix a positive integer $n$. 
If $v$ solves $\Lop_{n-1}[v]=0$ in $\D$, then $v$ is $n$-harmonic in $\D$.
\label{cor-imp}
\end{cor}

\begin{cor}
Fix a positive integer $n$. 
If $v$ is $n$-harmonic in $\D$, then $\Lop_{n-1}[u]$ is $(n-1)$-harmonic.
If $n=1$, this should be interpreted as $\Lop_0[u]=0$.
\label{cor-imp2}
\end{cor}

\subsection{Mapping properties of $\Lop_\theta$}
We need to understand what space $\Lop_\theta$ maps $\mathrm{PH}^p_{N,\alpha}(\D)$
into.

\begin{prop}
$(0<p<+\infty)$
Suppose $u\in\mathrm{PH}^p_{N,\alpha}(\D)$, for some integer 
$N=1,2,3,\ldots$ and a real parameter $\alpha$. Then, for $\theta$ real, we
have $\Lop_{\theta}[u]\in\mathrm{PH}^p_{N,\alpha+p}(\D)$. 
\label{prop-7.7}
\end{prop}

\begin{proof}
In view of Corollary \ref{cor-4}, we have
$\partial_z u,\bar\partial_z u\in\mathrm{PH}^p_{N,\alpha+p}(\D)$
and 
$$
\Delta u=4\partial_z\bar\partial_z u\in\mathrm{PH}^p_{N-1,\alpha+2p}(\D).
$$
Here, we used that $\Delta u$ is $(N-1)$-harmonic, which is immediate
because $u$ is $N$-harmonic. We remark that as a consequence, the function
$\Mop\Delta u(z)=(1-|z|^2)\Delta u(z)$ is $N$-harmonic, and then
%{\color{red}
$\Mop\Delta u\in\mathrm{PH}^p_{N,\alpha+p}(\D)$.
%}
The assertion now follows.
\end{proof}

\subsection{The work of Olofsson}
Given a function $v$ on $\D$, For $0<r<1$, we let $v_r$ denote its 
{\em dilate}, the function on $\Te$ given by $v_r(\zeta):=v(r\zeta)$.
Recently, Olofsson obtained the following result \cite{Olof2}, which is 
somewhat analogous to the classical theory of axially symmetric potentials
due to Weinstein et al. (see, e.g., \cite{Wein}).

\begin{prop}
$(0<p<+\infty,\,-\frac12<\theta<+\infty)$  
For a continuous function $v:\D\to\C$, the following are equivalent: 

\noindent{\rm(i)} The function $v$ solves 
%(in the sense of distribution theory)  
$\Lop_\theta[v]=0$ on $\D$, and its dilates $v_r$ converge to 
a distribution as $r\to1^-$, and

\noindent{\rm(ii)} The function $v$ is of the form 
\[
v(z)=\frac{1}{2\pi}\int_{\Te}
\frac{(1-|z|^2)^{2\theta+1}}{|1-z\bar\xi|^{2\theta+2}}f(\xi)\diff s(\xi),
\qquad z\in\D,\]
for some distribution $f$ on $\Te$, where the integral is understood in the 
sense of distribution theory. 
\label{prop-olof2}
\end{prop}

We now indicate how to derive Proposition \ref{prop-olof1} from the above
Proposition \ref{prop-olof2}.

\begin{proof}[Proof of Proposition \ref{prop-olof1}]
We realize that Proposition \ref{prop-olof1} follows from Proposition 
\ref{prop-olof2} once we know that adding the following requirements in 
the setting of Proposition \ref{prop-olof2}(i) forces the function $v$ to have
convergent dilates $v_r$ as $r\to 1^-$ in the sense of distribution theory:
(1) $\theta=n-1$ is a nonnegative integer, and (2) $v\in L^p_\alpha(\D)$. 
First, we observe that since 
%{\color{red}
$\Lop_{n-1}[v]=\Lop_\theta[v]=0$
%} 
is assumed to hold, we know from Corollary \ref{cor-imp} that
$v$ is $n$-harmonic, so that $v\in\mathrm{PH}^p_{n,\alpha}(\D)$. 
Next, by suitably iterating Corollary \ref{cor-5} and by combining the result 
with the pointwise bound of Corollary \ref{cor-2.6}, we realize that all the 
harmonic functions $u_j$ in the Almansi representation of  
$v\in\mathrm{PH}^p_{n,\alpha}(\D)$,
\[
u(z)=u_0(z)+|z|^2u_1(z)+\cdots+|z|^{2n-2}u_{n-1}(z),
\]
have the growth bound 
\[
|u_j(z)|=\mathrm{O}\big((1-|z|)^{-\Lambda}\big)\quad
\text{as}\,\,\,|z|\to1^-,
\]
for some big positive constant $\Lambda$ (which may depend on $n,p,\alpha$). 
By a standard result in distribution theory, this means that each $u_j$ is 
the Poisson integral of a distribution $f_j$ on $\Te$, and that the dilates 
%{\color{red}
$u_{j,r}(\zeta):=u_j(r\zeta)$ (for $\zeta\in\Te$) converge to a nonzero
constant multiple of $f_j$ in the sense of distribution theory as $r\to1^-$.
%, where the constant is. %$1/c(n-1)$ (see \eqref{eq-ctheta}). 
Then the dilates of $u_r$ of $u$ converge %to the same nonzero constant 
%multiple of $f:=f_0+f_1+\cdots+f_{n-1}$ as $r\to1^-$, again 
in the sense of 
distribution theory, as needed.
%}
\end{proof}

\subsection{The proof of the cellular decomposition theorem}

We turn to the proof of Theorem \ref{thm-entamain2}.

\begin{proof}[Proof of Theorem \ref{thm-entamain2}]
\textsc{Uniqueness}. We first treat the uniqueness part. We have the equation
\begin{equation}
w_0+\Mop[w_1]+\cdots+\Mop^{N-1}[w_{N-1}]=0,
\label{eq-uniq1}
\end{equation}
where the functions $w_j$ are $(N-j)$ harmonic and solve the partial 
differential equation 
\newline\noindent 
%{\color{red}
$\Lop_{N-j-1}[w_j]=0$
%} 
on $\D$. We need to show that all the functions $w_j$ vanish, where 
$j=0,\ldots,N-1$. We resort to an induction argument. Clearly, when $N=1$, 
the equation \eqref{eq-uniq1} just says $w_0=0$, as needed. Next, in 
the induction step we suppose that uniqueness
holds for $N=N_0$, and intend to demonstrate that it must then also hold for
$N=N_0+1$. The equation \eqref{eq-uniq1} with $N=N_0+1$ reads
\begin{equation*}
\sum_{j=0}^{N_0}\Mop^j[w_j]=0,
%\label{eq-uniq1}
\end{equation*} 
where the functions $w_j$ solve $\Lop_{N_0-j}[w_j]=0$. We now apply the 
operator $\Lop_{N_0}$ to both sides,
%\begin{equation*}
%\Lop_{N_0}[w_0]+\Lop_{N_0}\Mop[w_1]+\cdots+\Lop_{N_0}\Mop^{N_0}[w_{N_0}]=0,
%\label{eq-uniq1}
%\end{equation*} 
and rewrite the equation using \eqref{eq-operiden2}:
\begin{equation*}
\sum_{j=0}^{N_0}\big\{\Mop^j\Lop_{N_0-j}[w_j]+4j(j-2N_0-1)\Mop^{j-1}[w_j]\big\}=0.
%\label{eq-uniq1}
\end{equation*} 
If we use the given information that the functions $w_j$ solve 
$\Lop_{N_0-j}[w_j]=0$, the above equation simplifies pleasantly:
\begin{equation*}
\sum_{j=0}^{N_0-1}(j+1)(j-2N_0)\Mop^{j}[w_{j+1}]=0.
%\label{eq-uniq1}
\end{equation*} 
By introducing the functions $\tilde w_j:=(j+1)(j-2N_0)w_{j+1}$, the equation 
simplifies further: 
\begin{equation*}
\sum_{j=0}^{N_0-1}\Mop^{j}[\tilde w_{j}]=0.
%\label{eq-uniq1}
\end{equation*} 
We observe that $\Lop_{N_0-j-1}[\tilde w_j]=0$, and we are in the setting of 
%{\color{red}
$N=N_0$, and by the induction hypothesis, we have that $\tilde w_j=0$ for
all $j=0,\ldots,N_0-1$. As a consequence, $w_j=0$ for all $j=1,\ldots,N_0$,
%}
and $w_0=0$ is then immediate from \eqref{eq-uniq1}. 
This completes the uniqueness part of the proof.

\textsc{Existence}. Next, we turn to the existence part.
We argue by induction in $N$. We first observe that the assertion is trivial 
for $N=1$. In the induction step, we assume that the assertion of the 
theorem holds for $N=N_0\ge1$, and attempt to show that it must then also 
hold for $N=N_0+1$. 
The function $u$ is now $(N_0+1)$-harmonic, and and we form the
associated function
%\[
%\tilde w_{N_0}:=\frac{(-4)^{-N_0}}{(2N_0)!}\Lop_1\cdots\Lop_{N_0}[u].
%\]
%From Proposition \ref{prop-basicoe1}, we know that $\tilde w_{N_0}$ is 
%harmonic, since
%\[
%\Mop\Delta \tilde w_{N_0}=\Lop_0[\tilde w_{N_0}]=\frac{(-4)^{-N_0}}{(2N_0)!}
%\Lop_0\Lop_1\cdots\Lop_{N_0}[u]=\frac{(-4)^{-N_0}}{(2N_0)!}
%\Mop^{N_0+1}\Delta^{N_0+1}u=0,
%\] 
%where we used that $u$ is $(N_0+1)$-harmonic. Next, we claim that the 
%difference $u-\Mop^{N_0}[\tilde w_{N_0}]$ is $N_0$-harmonic. Indeed, 
%we see that
%\[
%\Mop^{N_0}\Delta^{N_0}\big(u-\Mop^{N_0}[\tilde w_{N_0}]\big)
%=\Lop_0\cdots\Lop_{N_0-1}
%\big[u-\Mop^{N_0}[\tilde w_{N_0}]\big]
%\]
$\Lop_{N_0}[u]$, which is then $N_0$-harmonic, by Corollary \ref{cor-imp2}.
Since $u\in\mathrm{PH}^p_{N_0+1,\alpha}(\D)$, Proposition \ref{prop-7.7} gives that
$\Lop_{N_0}[u]\in\mathrm{PH}^p_{N_0,\alpha+p}(\D)$. By the induction hypothesis,
then, we know that
\[
\Lop_{N_0}[u]=\sum_{j=0}^{N_0-1}\Mop^j[h_j],
\]
where $h_j$ are $(N_0-j)$-harmonic and solve $\Lop_{N_0-j-1}[h_j]=0$.
Moreover, again by the induction hypothesis, we have 
$\Mop^j[h_j]\in\mathrm{PH}^p_{N_0,\alpha+p}(\D)$, so that by 
\eqref{eq-multspaces2}, $h_j\in\mathrm{PH}^p_{N_0-j,\alpha+(1+j)p}(\D)$. We form 
the associated function
\[
H:=\frac14\sum_{j=0}^{N_0-1}\frac{1}{(j+1)(2N_0-j)}\Mop^{j+1}[h_j],
\]
and observe that $H\in\mathrm{PH}^p_{N_0+1,\alpha}(\D)$. So the sum $u+H$ is 
in $\mathrm{PH}^p_{N_0+1,\alpha}(\D)$, and we calculate that
\begin{multline*}
\Lop_{N_0}[u+H]=\Lop_{N_0}[u]+\Lop_{N_0}[H]=\sum_{j=0}^{N_0-1}\bigg\{\Mop^j[h_j]+
\frac{1}{4(j+1)(2N_0-j)}\Lop_{N_0}\Mop^{j+1}[h_j]\bigg\}
\\
=\sum_{j=0}^{N_0-1}\bigg\{\Mop^j[h_j]
+\frac{1}{4(j+1)(2N_0-j)}\big(\Mop^{j+1}\Lop_{N_0-j-1}[h_j]-4(j+1)(2N_0-j)
\Mop^{j}[h_j]\big)\bigg\}=0,
\end{multline*}
where we used the operator identity \eqref{eq-operiden2} and that 
$\Lop_{N_0-j-1}[h_j]=0$. 
So, with $w_0:=u+H$ and 
\[
w_j:=-\frac{1}{4j(2N_0-j+1)}h_{j-1},\qquad j=1,\ldots,N_0,
\]
we see that
\[
u=\sum_{j=0}^{N_0}\Mop^j[w_j],
\]
where $w_j$ is $(N_0+1-j)$-harmonic with $\Lop_{N_0-j}[w_j]=0$, for 
$j=0,\ldots,N_0$. Moreover, the given integrability properties of the functions 
$h_j$ lead to $w_j\in\mathrm{PH}^p_{N_0+1-j,\alpha+jp}(\D)$ and 
$\Mop^j[w_j]\in\mathrm{PH}^p_{N_0+1,\alpha}(\D)$. The existence of the asserted
expansion has now been obtained. The proof is complete.
\end{proof}

\section{Applications of the cellular decomposition}
\label{sec-applcellul}

\subsection{Triviality of individual terms in the cellular decomposition}
\label{subsec-triv}

We now analyze each term in the cellular decomposition separately.
We recall the definition of the functions $a_{j,N}(p)$ from Section 
\ref{sec-mainres}.

\begin{prop}
$(0<p<+\infty)$ Fix integers $N=1,2,3,\ldots$ and $j=0,\ldots,N-1$.
Suppose $u\in\mathrm{PH}^p_{N,\alpha}(\D)$ is of the form
$u=\Mop^j[w]$, where $w$ is $(N-j)$-harmonic and $\Lop_{N-j-1}[w]=0$. 
Then if $\alpha\le a_{N-j,N}(p)$, we have that $u=0$.
\label{prop-8.1}
\end{prop}

\begin{proof}
By \eqref{eq-multspaces2}, the function $w$ is in 
$\mathrm{PH}^p_{N-j,\alpha+jp}(\D)$. % and it solves $\Lop_{N-j-1}[w]=0$. 
From 
\eqref{eq-arel}, we know that
\[
\alpha+jp\le a_{N-j,N}(p)+jp=a_{N-j,N-j}(p),
\]
so if we introduce $N':=N-j$ and $\alpha':=\alpha+jp$, it will be sufficient 
to show the following: {\em If $w\in\mathrm{PH}^p_{N',\alpha'}(\D)$ solves 
$\Lop_{N'-1}[w]=0$, and $\alpha'\le a_{N',N'}(p)$, then $w=0$}. 
We recall from \eqref{eq-bj} and \eqref{eq-aj} that
\[
a_{N',N'}(p)=\min\{b_{N',N'}(p),-1\},\qquad 
b_{N',N'}(p)=\min\{-1-(2N'-1)p,-2+p\},
\]
so that 
\[
a_{N',N'}(p)=b_{N',N'}(p)=-1-(2N'-1)p\quad\text{for}\,\,\,\,
0<p\le\frac{1}{2N'}.
\]
The assertion $w=0$ is immediate from Proposition \ref{prop-5.6} in case 
$0<p\le1/(2N')$. Next, we consider the remaining interval $1/(2N')<p<+\infty$.  
Since
\[
a_{N',N'}(p)=\min\{p-2,-1\}\quad\text{for}\,\,\,\,\frac{1}{2N'}<p<+\infty,
\]
the assumption $\alpha\le a_{N',N'}(p)$ entails in view of Proposition 
\ref{prop-5.5} that $w=\Mop[\tilde w]$, where $\tilde w$ is $(N'-1)$-harmonic,
which means that $\tilde w=0$ if $N'=1$. 
In view of Theorem \ref{thm-entamain2}, the function $\tilde w\in
\mathrm{PH}^p_{N'-1,\alpha+p}(\D)$ then has
a unique expansion 
\[
\tilde w=\sum_{j=0}^{N'-2}\Mop^j[g_j],
\]
where $\Mop^j[g_j]\in\mathrm{PH}^p_{N'-1,\alpha+p}(\D)$ and $g_j$ is 
$(N-j-1)$-harmonic with $\Lop_{N'-j-2}[g_j]=0$. This means that $w=\Mop[\tilde w]$
has the expansion
\begin{equation}
w=\Mop[\tilde w]=\sum_{j=1}^{N'-1}\Mop^{j}[g_{j-1}]=\sum_{j=1}^{N'-1}\Mop^{j}[\tilde 
g_j],
\label{eq-uniQ2}
\end{equation}
where the terms are all in $\mathrm{PH}^p_{N',\alpha}(\D)$, and 
$\Lop_{N'-j-1}[\tilde g_j]=0$; here we have introduced $\tilde g_j:=g_{j-1}$. 
From the uniqueness of the cellular decomposition in Theorem 
\ref{thm-entamain2}, we realize that \eqref{eq-uniQ2} is only possible 
if $w=0$. The proof is complete.
\end{proof}

\begin{prop}
$(0<p<+\infty)$ Fix integers $N=1,2,3,\ldots$ and $j=0,\ldots,N-1$.
Suppose $\alpha$ is real with $\alpha>a_{N-j,N}(p)$. 
Then there exists a nontrivial $u\in\mathrm{PH}^p_{N,\alpha}(\D)$ of the form
$u=\Mop^j[w]$, where $w$ is $(N-j)$-harmonic with $\Lop_{N-j-1}[w]=0$. 
\label{prop-8.2}
\end{prop}

\begin{proof}
For $0<p\le 1$, we can use the function $u=U_{N-j,N}=\Mop^j[U_{N-j,N-j}]$, as given
by \eqref{eq-defUNk}, so that $w=U_{N-j,N-j}$. By Proposition \ref{prop-olof2}, 
the function $w=U_{N-j,N-j}$ solves $\Lop_{N-j-1}[w]=0$, 
%{\color{red}
by Lemma 
\ref{lem-6.2}, $u$ is in $\mathrm{PH}^p_{N,\alpha}(\D)$
%} 
if and only if $\alpha>b_{N-j,N}(p)$, and we have $a_{N-j,N}(p)=b_{N-j,N}(p)$ 
for $0<p\le 1$. For $1\le p<+\infty$, we need to consider instead the function
$u=\Mop^j[w]$, where
\begin{equation}
w(z)=(1-|z|^2)^{2N-2j-1}\int_{\Te}|1-z\bar\xi|^{-2N+2j}\frac{\diff s(\xi)}{2\pi},
\qquad z\in\D,
\label{eq-functw}
\end{equation}
which solves $\Lop_{N-j-1}[w]=0$, by a second application of Proposition 
\ref{prop-olof2}. The function $w$ given by \eqref{eq-functw} is bounded in
the disk $\D$, so that $u=\Mop^j[w]$ is in
$\mathrm{PH}^p_{N,\alpha}(\D)$ for
\[
\alpha>-1-pj=a_{N-j,N}(p),\qquad 1\le p<+\infty.
\]
The proof is complete. 
\end{proof}

\subsection{Characterization of the admissible region}
We are now ready to supply the proof of Theorem \ref{thm-main2}.

\begin{proof}[Proof of Theorem \ref{thm-main2}]
We first observe that
\[
\min_{j:1\le j\le N}a_{j,N}(p)=\min_{j:0\le j\le N}b_{j,N}(p).
\]
After that, we understand that the assertion is a consequence of the cellular 
decomposition of Theorem \ref{thm-entamain2} combined with Propositions
\ref{prop-8.1} and \ref{prop-8.2}. 
\end{proof}

\subsection{The entangled region}
It remains to supply the proof of Proposition \ref{prop-ent3.4}.

\begin{proof}[Proof of Proposition \ref{prop-ent3.4}] 
In terms of the cellular decomposition of Theorem \ref{thm-entamain2},
it is a matter of deciding for which $(p,\alpha)$ the function 
$w_{N-1}$ must equal $0$. This is easy to do using Propositions
\ref{prop-8.1} and \ref{prop-8.2}.  
\end{proof}

\subsection{The principal unentangled cell}
It remains to supply the proof of Proposition \ref{prop-snent3.5}.

\begin{proof}[Proof of Proposition \ref{prop-snent3.5}] 
In terms of the cellular decomposition of Theorem \ref{thm-entamain2},
it is a matter of deciding for which $(p,\alpha)$ the functions 
$w_j$, with $j=0,\ldots,N-2$, must all equal $0$. This is easy to do using 
Propositions \ref{prop-8.1} and \ref{prop-8.2}.
\end{proof}

\subsection{The cellular decomposition for the general 
admissible cell}
It remains to supply the proof of Theorem \ref{thm-entamain3}.

\begin{proof}[Proof of Theorem \ref{thm-entamain3}] 
The criteria of the theorem check which terms actually occur in the cellular 
decomposition of Theorem \ref{thm-entamain2}, in accordance with
Propositions \ref{prop-8.1} and \ref{prop-8.2}. 
\end{proof}

\section{The local weighted integrability problem and the classical
theorem of Holmgren}

\subsection{Uniqueness for the local weighted integrability criterion}

We now prove the local uniqueness theorem alluded to in Subsection 
\ref{subsec-local1}. We recall the definition of the ``pie'' $Q(J)$ for a
given arc $J$ of the circle $J$. Let $L^p_\alpha(\D)$ denote the weighted
Lebesgue space supplied with the (quasi-)norm
\[
\|u\|^p_{L^p_\alpha(\D)}=\int_{\D}|u(z)|^p(1-|z|^2)^\alpha\diff A(z)
<+\infty.
\]

\begin{thm}
$(0<p<+\infty)$
Let $J$ be an arc of $\Te$ of positive length $<2\pi$. Then the implication
\[
1_{Q(J)}u\in L^p_\alpha(\D)\implies u\equiv0
\]
holds for all $N$-harmonic functions on $\D$ if and only if 
$\alpha\le -(2N-1)p-1$. 
\label{thm-8.1.1}
\end{thm}

\begin{proof}
By considering a rotation of the function $U_{N,N}$ we realize that the 
condition is necessary for the implication to be valid.
%can easily get a local version of Proposition~\ref{prop-5.6}:
%If $J$ is a proper closed arc of the unit circle, and if $p>0$,
%$\alpha\le-1-(2N-1)p$, and
%$f\in\mathrm{PH}_{N}(\D)$ are such that
%$$
%\int_{z/|z|\in J}|f(z)|^p(1-|z|^2)^\alpha\,dA(z)<+\infty,
%$$
%then $f=0$. 

In the other direction, we build on the fact that the argument which gives 
Corollaries \ref{thm-2} and \ref{cor-3} may be essentially localized. This 
gives us the conclusion that
\[
1_{Q(J')}\partial_z^j\bar\partial_z^k u\in L^p_{\alpha+(j+k)p}(\D)
\]
for any arc $J'$ contained in $J$ whose two endpoints are different than
those of $J$. The rest of the argument mimics the proof of 
Proposition~\ref{prop-5.6}, except that we need to know that a holomorphic 
function $f$ cannot have
\[
\int_{Q(J')}
\frac{|f(z)|^p}{1-|z|}\diff A(z)<+\infty
\]  
unless $f$ vanishes identically. This can be obtained in a straight-forward
fashion, since the integrability requirement entails that
\[
\liminf_{r\to 1^-}
\int_{J'}|f(r\zeta)|^p\diff s(\zeta)=0,
\]
and the pointwise estimates which localize Corollary \ref{cor-2.6}
yield that $|f(z)|=\mathrm{O}((1-|z|)^{-1/p})$ holds near $J''$ for any
slightly smaller arc inside $J'$.
Since $\log|f|$ is subharmonic, an elementary estimate of harmonic measure 
can now be used to deduce that $f(z)\equiv0$. 
\end{proof}

%As a consequence, we obtain local analogs of Theorems~\ref{thm-main2},
%\ref{thm-main1}, \ref{thm-entamain2}, and \ref{thm-entamain3}.

\subsection{Remarks on other domains than the disk and 
constant-coefficient elliptic operators}
\label{subsec-otherdom}

In this subsection, we will need the concept of a Schwarz function; for
details, we refer to the monographs \cite{Dav}, \cite{Shap}. 
%The local 
%Schwarz $S(z)$ is associated with a real-analytically smooth curve $\Gamma$,
%and it is the locally holomorphic function with $S(z)=\bar z$ along $\Gamma$.  

For simplicity, we will focus on the biharmonic case $N=2$. 
Suppose that the disk $\D$ is replaced by a bounded simpy 
connected Jordan domain $\Omega$ with, say, $C^\infty$-smooth boundary. Let 
$\Gamma_\Omega(z,w)$ denote the Green function
for $\Delta^2$ on $\Omega$ with vanishing Dirichlet boundary data. 
Using Green's formula, we may represent any biharmonic function which is 
$C^2$-smooth up to the boundary in terms of integrals with respect to a
continuous density in $w$ along $\partial\Omega$ of the two Poisson-type 
kernels
\[
H_1(z,w):=\Delta_w\Gamma_\Omega(z,w),\quad H_2(z):=
\partial_{\mathrm{n}(w)}\Delta_w\Gamma_\Omega(z,w),
\]
for $w\in\partial\Omega$ and $z\in\Omega$. For the disk $\D$ it happens that
a certain non-trivial linear combination of $H_1,H_2$ is 
$\mathrm{O}(\mathrm{dist}(z,\partial\D)^3)$ except near the singularity $w$,
which is the maximal degree of flatness permitted by Holmgren's uniqueness
theorem (see the book of John \cite{John}).
Indeed, with $w=1$, this is the function $U_{2,2}(z)$ appearing in 
Section \ref{sec-6}. We do not expect such degeneracy for more general 
domains. For instance, when $\Omega$ is a non-circular ellipse it is 
impossible to find a non-trivial biharmonic function which decays like 
$\mathrm{O}(\mathrm{dist}(z,\partial\Omega)^3)$ even along an arc of the 
boundary $\partial\Omega$. 
Indeed, we have the following result (cf. \cite{Hed3}). 
%For completeness, we
%supply the (short) proof here.

\begin{thm}
Suppose $\Omega$ is a Jordan domain in the complex plane $\C$, and that
$I\subset \partial\Omega$ is a nontrivial arc (so that it contains more than 
one point), which is real-analytically
smooth. Let $S(z)$ be the local Schwarz function, that is, the function which
is holomorphic near $I$ with $S(z)=\bar z$ on $I$. Let $F$ be a biharmonic 
function $F$ in $\Omega$, which extends $C^2$-smoothly to $\Omega\cup I$, 
whose partial derivatives of order at most $2$ extend to vanish on $I$.
If $F$ does not vanish identically on $\Omega$, then the local Schwarz 
function $S(z)$ extends meromorphically to $\Omega$.
\label{thm-SF1}
\end{thm}

\begin{proof}
%\begin{equation}
%F(z)=\mathrm{o}(\mathrm{dist}(z,\partial\Omega)^2)\quad\text{as}\,\,\,
%\Omega\ni z\to I,
%\label{eq-decayF}
%\end{equation}
%where $z\to I$ is understood in the sense that $\mathrm{dist}(z,I)\to0$. 
%Now, if $\partial_z$ is the complex differential operator defined in 
%Subsection
%\ref{subsec-addlnot}, 
In the indicated setting, it is well-known that the local Schwarz function 
exists as a function holomorphic in a neighborhood of $I$ with $S(z)=\bar z$
on $I$.
We form the function $H(z):=\Delta F(z)$, which is harmonic in $\Omega$ and
extends continuously to $\Omega\cup I$, with $H|_I=0$. By Schwarzian reflection,
the function $H$ extends harmonically across $I$, in such a manner that in a 
neighborhood of the arc $I$, $H(z)=-H(\bar S(z))$. We may now think of $F$ as
the solution to the Cauchy problem $\Delta F=H$ 
and $F|_I=\partial_{\mathrm{n}} F|_I=0$ [here, $\partial_{\mathrm{n}}$ denotes the 
exterior normal derivative). The Cauchy-Kovalevskaya theorem tells us that
there exists a real-analytic solution to this Cauchy problem in a neighborhood 
of $I$, and by Holmgren's theorem, the solution must be unique (see, e.g., the
book of John \cite{John}). In conclusion,
$F$ extends real-analytically across $I$.   

In a second step, we form the function $G:=\partial_z^2 F$, which is then 
bianalytic in $\Omega$, as 
\[
16\bar\partial_z^2\partial_z^2F(z)=\Delta^2 F(z)=0
\] 
in $\Omega$, while $G(z)=0$ holds for $z\in I$. By an Almansi-type expansion, 
$G$ has the form $G(z)=G_0(z)+\bar zG_1(z)$, where $G_0,G_1$ are holomorphic 
in $\Omega$ and continuous in $\Omega\cup I$. Actually, by the previous 
argument, we know more: $G_0,G_1$ both extend holomorphically across $I$. 
Since $G(z)=0$ on $I$, we conclude that $G_0(z)=-\bar zG_1(z)=-S(z)G_1(z)$ on 
$I$, so that by the uniqueness theorem for holomorphic functions, 
$G_0(z)=-S(z)G_1(z)$ holds on a neighborhood of $I$. Next, unless 
$G(z)\equiv0$, this means that the Schwarz function $S(z)$ has a meromorphic 
extension to all of $\Omega$ given by $-G_0(z)/G_1(z)$. 
This argument does not apply if $G(z)\equiv0$. In this remaining case, we 
realize that
$\bar F$ is a nontrivial bianalytic function, which vanishes on $I$. 
So the same argument applied to $\bar F$ in place of $G$  gives us a local 
Schwarz function $S(z)$ which extends meromorphically to $\Omega$. 
\end{proof}

%In \cite{Hed3}, the following corollary of Theorem \ref{thm-SF1} was obtained.
The following corollary should be compared with the classical uniqueness 
theorem of Holmgren (cf. \cite{Hed3}).

\begin{cor}
Let $\Omega$ be the domain interior to an ellipse, which is not a circle.
Moreover, let $I$ be a nontrivial arc of $\partial\Omega$. If $F$ is a
biharmonic function in $\Omega$ which extends to a $C^2$-smooth function 
on $\Omega\cup I$, whose partial derivatives of order at most $2$ vanish
along $I$, then $F(z)\equiv0$. 
\end{cor}

\begin{proof}
It is well-known that the Schwarz function for the ellipse develops a branch 
cut along the segment between the focal points (cf. \cite{Dav}, \cite{Shap}), 
so it cannot in particular be meromorphic in $\Omega$. So, in view of Theorem 
\ref{thm-SF1}, we must have $F(z)\equiv0$, as claimed. 
\end{proof}

\begin{rem}
If we apply a suitable affine transformation of the plane, the ellipse turns 
into $\D$, but the Laplacian changes to a related constant-coefficient 
elliptic operator of order $2$. The above example of the ellipse strongly 
suggests that the results of this paper do not carry over to the more 
general setting of elliptic constant-coefficient operators of order 
$2N$ on the disk $\D$, even when we have the $N$-th power of a given elliptic 
constant-coefficient operator of order $2$ (for $N>1$). Indeed, for $N=2$,
this follows from the results presented below in Subsection 
\ref{subsec-8.extra}, since already the local uniqueness problem finds a 
different solution. 
\end{rem}

%\begin{rem}
%We should probably mention that the assertion of Theorem \ref{thm-SF1}
%does not fully capture the rigidity imposed the vanishing of all partial 
%derivatives of order at most $2$ on $I$. In fact, the proof uses only that
%the values and the second order partial derivatives vanish on $I$.
%For domains with a meromorphic Schwarz function, additional information 
%comes from the vanishing of the first-order derivatives. 
%If -- for simplicity -- the function $F$ is real-valued, then we decompose 
%$F(z)=F_0(z)+|z|^2F_1(z)$ (this is called the Almansi expansion, see Subsection
%\ref{subsec-Almansi1} below), where $F_0,F_1$ are harmonic in $\Omega$. 
%It is then possible to obtain from the additional vanishing of the 
%first-order partial derivatives that the function 
%$F_1(\bar S(z))$ extends harmonically to the real part of a meromorphic 
%function in $\Omega$. We note that 
%\[
%G(z)=\partial^2_zF(z)=\partial_z^2F_0(z)+\bar z
%(2\partial_z F_1(z)+z\partial_z^2F_1(z)),  
%\]
%so that the pairs $F_0,F_1$ and $G_0,G_1$ are connected by
%$G_0(z)=\partial_z^2F_0(z)$ and 
%$G_1(z)=2\partial_z F_1(z)+z\partial_z^2F_1(z)$.
%\end{rem}

\subsection{The weighted integrability of biharmonic functions on an
ellipse}
\label{subsec-8.extra}

It is a natural question to ask what happens with the local uniqueness
in Theorem \ref{thm-8.1.1} when the circle is replaced by an ellipse.
Let $\Omega$ be the interior to an ellipse, which is assumed not to be a 
circle. We let $L^p_\alpha(\Omega)$
consist of the functions $u$ with
\[
\int_\Omega|u(z)|^p[\mathrm{dist}(z,\partial\Omega)]^\alpha\diff A(z)<+\infty.
\]
For an arc $I$ of the ellipse $\partial\Omega$, we let $Q(I)$ denote the 
corresponding ``pie'', whose boundary consists of two linear rays from the
center of the ellipse to the endpoints of $I$, together with the arc $I$. 

\begin{thm}
Let $\Omega$ be the domain interior to an ellipse, which is not circular.
Moreover, let $I\subset\partial\Omega$ be a nontrivial boundary arc, such
that the complementary arc $\partial\Omega\setminus I$ is also nontrivial.
Then the implication
\[
1_{Q(I)}u\in L^p_\alpha(\Omega)\implies u\equiv0
\]
holds for all biharmonic functions $u$ on $\Omega$ if and only if 
$\alpha\le-2p-1$. 
\label{thm-8.2.2}
\end{thm}

\begin{proof}
By solving the Dirichlet problem for the bilaplacian with nontrivial smooth
boundary data (but with data that vanish on $I$), we may obtain a nontrivial 
biharmonic function $u$ on $\Omega$ which decays like 
$u(z)=\mathrm{O}(\mathrm{dist}(z,\partial\Omega)^2)$ near the arc $I$.
This function $u$ will have $1_{Q(I)}u\in L^p_\alpha(\Omega)$ for  
every $\alpha>-2p-1$. This settles the ``only if'' part. 

We now turn to the ``if'' part. So we assume $\alpha\le-1-2p$ and try to
show that the asserted implication holds. 
As in the proof of Theorem \ref{thm-8.1.1}, we observe that we may obtain
from the technique in Corollaries \ref{thm-2} and \ref{cor-3} that 
\[
1_{Q(I')}\partial_z^j\bar\partial_z^k u\in L^p_{\alpha+(j+k)p}(\Omega),
\]
for any arc $I'$ contained in $I$ whose endpoints are different.
We form the function $G(z):=\partial_z^2u(z)$, which is bianalytic in $\Omega$
with $1_{Q(I')}G\in L^p_{\alpha+2p}(\Omega)$. The function $G$ is bianalytic,
so that $G(z)=G_1(z)+\bar zG_2(z)$, where $G_1,G_2$ are holomorphic. 
We know also that $G_2(z)=\bar\partial_z\partial_z^2 u(z)$ has $1_{Q(I')}G_2
\in L^p_{\alpha+3p}(\Omega)$. We let $S(z)$ be the Schwarz function for the
boundary ellipse $\partial\Omega$, which has the property that $\bar z-S(z)=
\mathrm{O}(\mathrm{dist}(z,\partial\Omega))$. We conclude that if $I'$ is made
smaller so that the ``pie'' $Q(I')$ avoids the foci of the
ellipse, the function $1_{Q(I')}(z)(S(z)-\bar z)G_2(z)$ is in 
$L^p_{\alpha+2p}(\Omega)$. Next, since $G(z)=G_1(z)+\bar z G_2(z)$ is in 
$L^p_{\alpha+2p}(\Omega)$, we conclude that 
$1_{Q(I')}(z)\{G_1(z)+S(z)G_2(z)\}$ is in 
$L^p_{\alpha+2p}(\Omega)\subset L^p_{-1}(\Omega)$. The implied integral decay 
of the holomorphic function $G_1+SG_2$ is too strong near the arc $I'$, 
which leads to the conclusion that $G_1(z)+S(z)G_2(z)\equiv0$ (an argument 
can be based on harmonic measure estimates; compare
with the proof of Theorem \ref{thm-8.1.1}). We are left with 
$\partial_z^2u(z)=G_1(z)+\bar z G_2(z)=(\bar z-S(z))G_2(z)$ which is possible
only if $G_2(z)\equiv0$ as a result of the branch cuts for $S(z)$ at the foci
of the ellipse. Finally, if $G_2(z)\equiv0$ then $G(z)\equiv0$ and $\bar u$ is
bianalytic. Arguing with $\bar u$ in place of $G$ settles the ``if'' part.
The proof of the theorem is complete. 
\end{proof}

\begin{rem}
We see from a comparison of Theorems \ref{thm-8.1.1} and \ref{thm-8.2.2} 
that circles and non-circular ellipses behave quite differently. 
\end{rem}

\begin{opprob}
It remains an open problem whether the critical integrability type -- defined
in an analogous fashion -- for biharmonic functions ($N=2$) on a bounded 
simply connected quadrature domain $\Omega$ is the same as for the disk or not
(compare with Theorem \ref{thm-main1}). Even the analogue of the local theorem
(Theorem \ref{thm-8.1.1}) appears unknown in the context of quadrature domans. 
\end{opprob}

\section{Concluding remarks}

\subsection{Boundary effects}

It would be interesting to find out necessary and sufficient 
conditions on a distribution $f$ on $\Te$ in order to have that the
potential
\[
\mathbf{U}_{N,N}[f](z)=(1-|z|^2)^{2N-1}\int_\Te|1-z\bar\xi|^{-2N}
f(\xi)\,\frac{\diff s(\xi)}{2\pi},
\]
belongs to, 
%{\color{red}
say, $\mathrm{PH}^p_{N,\alpha}(\D)$ for $0<p<1/(2N)$
%}
and $\alpha>b_{N,N}(p)$ close to $b_{N,N}(p)$. It is obvious that a sum of 
point masses at points of $\Te$ with coefficients from $\ell^p$ gives rise 
to such a distribution $f$. However, it is easy to see that 
the functions $U_{N,N}(\gamma z)$ depend continuously on 
$\gamma\in\Te$ in the space $\mathrm{PH}^p_{N,\alpha}(\D)$, so the correct 
answer is not the trivial one (sums of point masses at points of $\Te$ with 
coefficients from $\ell^p$). 

%For example, for $n=2$, $p=\frac 14$, $\beta=1+2p=\frac 32$ 
%our question is to describe the distributions $\mu$ on $\mathbb T$ such that
%\begin{multline*}
%\int_{\mathbb D}\Bigl|\int_{\mathbb T}\frac{(1-|z|^2)^3\,d\mu(\zeta)}
%{|1-z\bar\zeta|^4}\Bigr|^p\,\frac{dm_2(z)}{(1-|z|^2)^\beta}
%\\=\int_{\mathbb D}\Bigl|\int_{\mathbb T}\frac{d\mu(\zeta)}
%{|1-z\bar\zeta|^4}\Bigr|^{1/4}\,\frac{dm_2(z)}{(1-|z|^2)^{3/4}}<\infty.
%\end{multline*}

\subsection{Sharpening of some results}
It is possible to sharpen the assertion of Theorem \ref{thm-main2} 
to arrive at the following results (cf. Section 4 of \cite{Ale}).

\noindent{\rm(i)} 
Suppose $0<p<+\infty$, and that $u$ is $N$-harmonic in $\D$. If 
\[
\liminf_{r\to1^-}\,\,\,(1-r)^{\beta(N,p)}\int_{\D\setminus\D(0,r)}
|u|^p\,\diff A=0,
\]
then $u(z)\equiv0$.

\noindent{\rm(ii)} 
If $0<p<\frac1{2N-1}$, and if 
\[
\liminf_{r\to 1^-}\frac{\int_{\D\setminus\D(0,r)}|u|^p\diff A}
{\int_{\D\setminus\D(0,r)}|U_{N,N}|^p\,\diff A}=0,
\]
then $u(z)\equiv0$.
\medskip

Here, it is useful to know that
\[
\int_{\D\setminus\D(0,r)}|U_{N,N}(z)|^p\diff A(z)\asymp 
\begin{cases}
(1-r)^{2-p},\qquad \text{if}\,\,\,p>\frac1{2N},\\
(1-r)^{2-p}\log\frac1{1-r},\qquad  \text{if}\,\,\,p=\frac1{2N},\\
(1-r)^{1+(2n-1)p},\qquad \text{if}\,\,\,0<p<\frac1{2N}.
\end{cases}
\]
%In connection with the above sharpening of our results, it appears natural 
%to ask for a characterization of the distributions $f$ on $\Te$ for which 
%$$
%\int_{\D\setminus\D(0,r)}|\mathbf{U}_{N,N}[f](z)|^p\diff A(z)
%=\mathrm{O}\bigg(\int_{\D\setminus\D(0,r)}|U_{N,N}(z)|^p\diff A(z)\bigg),
%\qquad \text{as}\,\,\,r\to 1^-.
%$$
%Already the harmonic case $N=1$ is quite nontrivial.
%, $n=1$, $0<p<1/2$, 
%there are distributions $\mu$ which are not (finite) measures such that 
%$$
%\int_{r<|z|<1}|\mu*u_1(z)|^p\,dm_2(z)
%=O\bigl((1-r)^{1+p}\bigr),\qquad r\to 1.
%$$
%{\color{red}

\subsection{
%{\color{red}
Remarks on the weighted integrability of polyanalytic functions
%}
}

In the polyanalytic case the corresponding critical exponent 
is $\beta=-1-(N-1)p$, which may be interpreted as saying that no 
entanglement takes place. 
To be more precise, let $\mathrm{PA}^p_{N,\alpha}(\D)$ 
denote the subspace of $L^p_\alpha(\D)$ consisting of $N$-analytic functions,
which by definition solve the partial differential equation 
$\bar\partial_z^Nf=0$ on $\D$. 
Then 
\[
\mathrm{PA}^p_{N,\alpha}(\D)=\{0\} 
%\Longleftrightarrow
\,\,\,\iff \,\,\,\alpha\le -1-(N-1)p.
\]
We explain the necessary argument for $N=2$. 
Then $f\in\mathrm{PA}^p_{2,\alpha}(\D)$ 
decomposes into $f(z)=f_1(z)+\bar z f_2(z)$ with holomorphic $f_1,f_2$, 
so that
\[
zf(z)=zf_1(z)+|z|^2f_2(z)
\] 
is biharmonic. We form the extension of $zf$,
$\mathbf{E}[zf](z,\varrho)=zf_1(z)+\varrho^2f_2(z)$,
so that $\mathbf{E}[zf](z,\varrho)=zf(z)$ for $z\in\Te(0,\varrho)$.
Now, if 
\[
\int_{\D}
|zf(z)|^p(1-|z|^2)^{-1-p}\diff A(z)<+\infty,
\]
an elementary argument shows that
\[
\liminf_{\varrho\to1^-}\,\,\,(1-\varrho^2)^{-p}\int_{\Te(0,\varrho)}
|\mathbf{E}[zf](\zeta,\varrho)|^p\diff s(\zeta)
=\liminf_{\varrho\to1^-}\,\,\,(1-\varrho^2)^{-p}\int_{\Te(0,\varrho)}
|\zeta f(\zeta)|^p\diff s(\zeta)
=0.
\]
As the function $|\mathbf{E}[zf](\cdot,\varrho)|^p$ is subharmonic, 
we have that
\[
|\mathbf{E}[zf](z,\varrho)|^p
\le\frac{1}{2\pi\varrho}\frac{\varrho+|z|}{\varrho-|z|}
\int_{\Te(0,\varrho)}|F_\varrho(\zeta)|^p\diff s(\zeta),
\qquad z\in\D(0,\varrho),
\]
and a combination of the above tells us that $\mathbf{E}[zf](z,1)
=zf_1(z)+f_2(z)\equiv0$.
We arrive at $zf(z)=z(1-|z|^2)f_1(z)$, and the problem reduces
to $N=1$, $\alpha\le-1$, which is trivial.

\end{document}